\documentclass[preprint,12pt]{elsarticle}

\usepackage{algorithm}
\usepackage{algorithmic}
\usepackage{amsmath,amssymb,amsfonts,graphics,
latexsym, amscd, amsfonts, epsfig, eepic,epic,psfrag,setspace,eufrak}
\usepackage[all]{xypic}
\usepackage{color}

\usepackage{amssymb}
\usepackage{amsthm}

\theoremstyle{plain}
\newtheorem{theorem}{Theorem}
\newtheorem{corollary}{Corollary}

\newtheorem{lemma}{Lemma}
\newtheorem{proposition}{Proposition}
\theoremstyle{definition}
\newtheorem{definition}{Definition}

\theoremstyle{remark}

\journal{Discrete Mathematics}

\begin{document}

\begin{frontmatter}



\title{A topological framework for signed permutations}


\author[dk]{Fenix W.D.\ Huang}
\ead{fenixprotoss@gmail.com}

\author[dk]{Christian M.\ Reidys\corref{cor1}}
\ead{duck@santafe.edu}

\cortext[cor1]{Corresponding author}

\address[dk]{Department of Mathematic and Computer science, University of
    Southern Denmark, Campusvej 55, DK-5230 Odense M, Denmark}

\begin{abstract}
In this paper we present a topological framework for studying signed
permutations and their reversal distance. As a result we can give
an alternative approach and interpretation of the Hannenhalli-Pevzner
formula for the reversal distance of signed permutations. Our approach
utlizes the Poincar\'e dual, upon which reversals act in a particular
way and obsoletes the notion of ``padding'' of the signed permutations.
To this end we construct a bijection between signed permutations and
an equivalence class of particular fatgraphs, called $\pi$-maps, and
analyze the action of reversals on the latter. We show that reversals
act via either slicing, gluing or half-flipping of external vertices,
which implies that any reversal changes the topological genus by at
most one. Finally we revisit the Hannenhalli-Pevzner formula employing
orientable and non-orientable, irreducible, $\pi$-maps.
\end{abstract}


\end{frontmatter}

\section{Introduction}\label{S:1}

In a seminal paper Hannenhalli and Pevzner \cite{Pevzner:99} give a polynomial time
algorithm as well as an explicit formula to compute the reversal distance of
a signed permutation $b_n$. Employing the framework of breakpoint graphs
\cite{Watterson:1982,Nadeau:1984,Bafna:1996},
they express the reversal distance as
\begin{equation}\label{E:kk}
d(b_n) =
\begin{cases}
b-c+h+1 \quad \text{\small if $h\ge 3$, $h\equiv 1\mod 2$, all $h$ hurdles are super-hurdles,} \\
b-c+h \quad \text{\small otherwise.}  \\
\end{cases}
\end{equation}
where $b$ is the number of breakpoints, $c$ is the number of cycles and $h$ is the
number of hurdles in the breakpoint graph of the signed permutation \cite{Pevzner:99}.
The algorithm was implemented with time complexity $O(n^4)$ ($O(n^5)$ with padding) in
\cite{Pevzner:99} and improved later by \cite{Kaplan:97} to $O(n^2)$ ($O(n^3)$ with padding).
A linear time algorithm for finding the reversal distance of signed permutation is given by
\cite{Bader:2001}.
Subsequent work on analyzing this question for unsigned permutations however, has not
yet succeeded, and only approximations can be found
\cite{Bafna:1996,Caprara:1997,Christie:1998,Kececioglu:95}.

In this paper we present a topological framework for studying signed permutations,
representing an alternative to the breakpoint graph. We believe that this framework can
be adapted for studying the transposition distance of unsigned permutations.
To understand this problem one has to study the action of transpositions on orientable
cell-complexes. Accordingly, the reversal distance of signed and the transposition
distance of unsigned permutations become closely related problems. The former can be
described by reversals acting on cell-complexes of non-orientable surfaces and the
latter by transpositions acting on cell-complexes of orientable surfaces.

Specifically, we construct from a signed permutation an equivalence class of particular
fatgraphs \cite{Penner:10}, called $\pi$-maps. While several features of $\pi$-maps can
be found also in breakpoint graphs, like for instance oriented cycles of external
vertices, the key difference lies in considering the Poincar\'e dual, which comes
natural within the topological framework.

$\pi$-maps offer a combinatorial interpretation for the action of reversals on signed
permutations. The combinatorial interpretation has several implications: first there
is no need to introduce any ``padding'' \cite{Pevzner:99}, i.e.~inflating the underlying
signed permutation into a specific one, having the same reversal distance. In
\cite{Pevzner:99} it is stipulated that padding can be avoided, but all constructions
are based on padded configurations, i.e.~are not directly applied to the original,
signed permutation. Secondly the proofs become very intuitive, a consequence of passing
to the Poincar\'e dual in which reversals have only  a ``local'' effect that has concrete,
combinatorial interpretation. To be explicit, the action of reversals does not relocate
the sectors around the center vertex: its does only affect their orientations.

Cell-complexes of non-orientable surfaces are typically studied using the
orientational double-cover \cite{Penner:10}. The double-cover allows to
mimic the permutation framework of the orientable cell complexes and fits
therefore into the notions of half-edges and fatgraphs \cite{Penner:10}.
We shall, however, adopt here a different point of view: we base our definition
of fatgraph on the notion of sectors, which is a pair of half-edges, together
with an orientation. While this allows to reduce everything to permutations of
sectors, one has to give up the fact that the fixed-point free involution
algebraically relates vertices and boundary components, see Section~\ref{S:2b}.

The paper is organized as follows: first we recall in Section~\ref{S:2b} some
basic facts on signed permutations and on fatgraphs.
In Section~\ref{S:pi} we construct from a signed permutation a $\pi$-map. This
association is, however, not unique. This gives rise to a bijection between
signed permutations and certain equivalence classes of $\pi$-maps. We then
proceed analyzing basic properties of $\pi$-maps.

In Section~\ref{S:irreducible} we characterize irreducibility and components of
$\pi$-maps. We study orientable and non-orientable components, characterized by
the combinatorial $\sigma$-crossing.

We then analyze in Section~\ref{S:reversal} the action of reversals on $\pi$-maps
and show that reversals act
by either slicing, gluing or half-flipping of vertices. As a result the action
of reversals is Lipschitz continuous with respect to topological genus and implies
that the genus of the $\pi$-map is a sharp lower bound for the reversal distance.

In Section \ref{S:genus} we revisit a result of \cite{Pevzner:99} concerning the
successive breakdown of non-orientable components. As a consequence topological
genus is a sharp lower bound for the reversal distance.

In Section~\ref{S:orientable} we collect some facts about the action of reversals
on a set of orientable components.

Finally, we reformulate in Section~\ref{S:hurdles} the treatment of hurdles of
\cite{Pevzner:99} into the language of $\pi$-maps and give an interpretation of
the Hannenhalli-Pevzner formula, eq.~(\ref{E:kk}). Namely, two of the three terms
express just the topological genus of the underlying $\pi$-map.

\section{Signed permutations and fatgraphs}\label{S:2b}

Let $S_n$ denote the symmetric group over $[n]$. A permutation is a
one-to-one mapping $x\colon [n]\longrightarrow [n]$ and represented
as an $n$-tuple $ x=[x_1,\dots,x_n] $, where $x_i=x(i)$. Furthermore,
let $\varepsilon=[\varepsilon_1,\dots,\varepsilon_n] \in \{-1,+1\}^n$,
the $n$-tuple of ``signs''.

A signed permutation is a pair $b_n=[\varepsilon,x]=[\varepsilon_1 x_1,
\dots,\varepsilon_n x_n]$ and we denote the set of signed
permutations by $\mathfrak{B}_n$. Clearly we have $\vert
\mathfrak{B}_n\vert =2^nn!$ and $\mathfrak{B}_n$ carries a natural
structure of a group via
\begin{equation*}
[\varepsilon_x,x]\cdot [\varepsilon_y,y]= [\varepsilon_x\cdot
\varepsilon_y^{x},x\cdot y],\quad \text{\rm where}\
[\varepsilon_y^{x}]_i=[\varepsilon_y]_{x(i)}.
\end{equation*}
That is, there is an additional action on $\varepsilon_y$ when
commuting it with $x$, given by the $x$-permutation of the
coordinates. A reversal $\rho_{i,j}$ is the particular signed
permutation:
\begin{equation*}
\rho_{i,j} = [\varepsilon^\rho_{i,j},(1,\dots, i-1, j, j-1, \dots,
i, j+1,\dots)],
\end{equation*}
where
\begin{equation*}
(\varepsilon^\rho_{i,j})_h=
\begin{cases}
-1 & \text{\rm for $i\le h\le j$}\\
+1 & \text{\rm otherwise.}
\end{cases}
\end{equation*}
Accordingly, a reversal $\rho_{i,j}$ acts (via right-multiplication)
on $B_n$ as follows:
\begin{eqnarray*}
[\varepsilon_1x_1,\dots,\varepsilon_ix_i,\dots,\varepsilon_jx_j,
\dots,\varepsilon_nx_n]\cdot \rho_{i,j} = \\ \text{}
[\varepsilon_1x_1,\dots,-\varepsilon_jx_j,\dots,-\varepsilon_ix_i,
\dots, \varepsilon_nx_n].
\end{eqnarray*}
$\rho_{i,j}$ transforms the subsequence $(\varepsilon_ix_i,\dots,\varepsilon_j
x_j)$ into $(-\varepsilon_jx_j,\dots,-\varepsilon_ix_i)$ by inverting order
and signs within the interval $[i,j]$, for instance
\begin{eqnarray*}
[-5,+1,-3,+2,+4,+6] \cdot \rho_{3,4} & = &
[-5,+1,-2,+3,+4,+6] \\ \text{}
[-5,+1,-2,+3,+4,+6] \cdot \rho_{3,3} & = & [-5,+1,+2,+3,+4,+6].
\end{eqnarray*}

\begin{figure}[t]
\begin{center}
  \includegraphics[width=0.4\columnwidth]{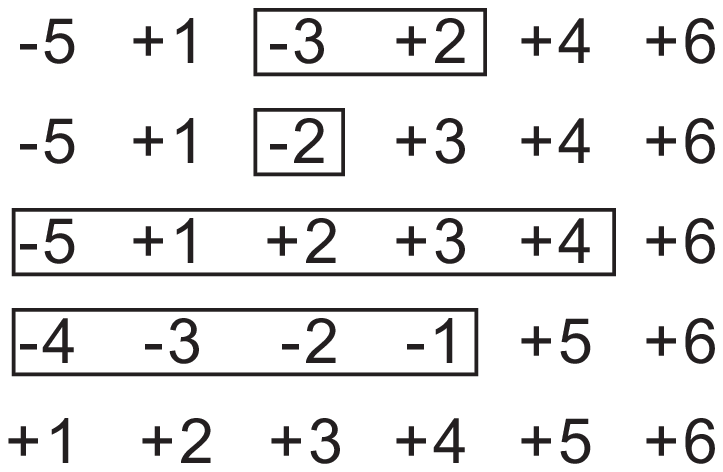}
\end{center}
\caption{\small
Signed permutations and the action of reversals. The signed permutation
$[-5, +1, -3, +2, +4, +6]$ is transformed into the identity via the actions
of the four reversals $\rho_{2,3}$, $\rho_{3,3}$, $\rho_{1,5}$ and $\rho_{1,4}$.
}
\label{F:permutation}
\end{figure}


A sector $x=(\lambda_x,\omega_x)$ is a pair consisting of a label
$\lambda_x\in \{1,\ldots,2n\}$ and an orientation denoted by $\omega_x\in
\{+,-\}$. We may depict $x$ as a labeled, oriented wedge, composed by an in-
and out-half edge. We denote counterclockwise and clockwise orientations
by $\omega_x=+$ and $\omega_x=-$, respectively. By abuse of notation we shall also
write a sector alternatively as $\pm \lambda_x$. By abuse of notion, we shall refer to
an counterclockwise oriented sector $x=(\lambda_x, +)$ as $\lambda_x$ and a clockwise
oriented sector $x=(\lambda_x, -)$ as $-\lambda_x$.

\begin{definition}
A fatgraph is a triple $\mathbb{G}=(H_{2n},\sigma,\gamma)$ where
$H_{2n}=\{1,\ldots,2n\}$ is a set of labeled sectors and $\sigma,\gamma$
are permutations of sectors such that to each pair $(x,\sigma(x))$ there
exists a unique $(y,\sigma(y))$ such that
$$
\diagram
x \dline^\gamma\rto^\sigma & \sigma(x) \\
\sigma(y)   & y\lto^\sigma\uline^\gamma
\enddiagram
\quad \text{\rm or }\quad
\diagram
x \drline^\gamma\rto^\sigma & \sigma(x) \\
\sigma(y)\urline_\gamma   & y\lto^\sigma
\enddiagram
$$
The directions of the $\gamma$-verticals are implied by the orientations of
pairs of sectors $(x,\sigma(x))$ and $(y,\sigma(y))$ and we shall refer
to the above diagrams as {\it untwisted} and {\it twisted} ribbons, respectively.
The genus of a fatgraph $\mathbb{G}$ is the genus of its underlying
topological quotient space, obtained by identifying the $\gamma$-sides of
the ribbons.
\end{definition}

The cycles of the permutation $\sigma$ are called vertices, $v$, i.e.~a vertex is
a cycle of sectors. As in the case of orientable fatgraphs \cite{Penner:10} we
follow the convention that vertex-cycles are traversed counterclockwise.
Note that $\sigma$ and $\gamma$ are unsigned permutations of sectors.

A cycle $\gamma_1=(s^1_1,\ldots,s^1_k)$ of $\gamma$ can be depicted as to connect
the out-half edge of $s^1_i$ with the in-half edge of $s^1_{i+1}$. By construction
this does not imply that successive sectors have equal orientations, since traversing
twisted ribbons changes the latter. A cycle of $\gamma$ is called a boundary component.
A fatgraph is called {\it unicellular} if $\gamma$ is an unique cycle.

A ribbon can be denoted by $((x,\sigma(x)),(y,\sigma(y))$. For untwisted ribbons we
have $\omega_x=\omega_{\sigma(y)}$ and $\omega_{\sigma(x)}=\omega_{y}$, while for twisted
ribbons $\omega_{x}=-\omega_{y}$ and $\omega_{\sigma(x)}=-\omega_{\sigma(y)}$ holds.
Ribbons with mono- and bi-directional verticals are called $m$- and $b$-ribbons,
respectively. For $m$-ribbons we have $\omega_x=-\omega_{\sigma(x)}$ (as well as
$\omega_y=-\omega_{\sigma(y)}$) and for $b$-ribbons $\omega_x=\omega_{\sigma(x)}$ (as well
as $\omega_y=\omega_{\sigma(y)}$).

A flip of a vertex $v$ is obtained by reversing the cyclic ordering of the sectors
incident on $v$ and changing their respective orientations. This is tantamount to
replacing any untwisted and twisted ribbon, that is incident to another vertex by
a twisted and untwisted ribbon, respectively. Furthermore loops remain unchanged.
Flipping does not affect the underlying topological quotient space, see
Fig.~\ref{F:flip} (B). Indeed, by the fundamental structure theorem of surfaces
\cite{Massey:69}, the latter depends only on the relative directions of the sides of ribbons
which are unaffected by flipping. In Fig.~\ref{F:flip} we illustrate the concept of
fatgraphs and vertex-flips.

\begin{definition}\label{D:iso}
Two fatgraphs $\mathbb{G}_1=(H_{2n}^1,\sigma^1,\gamma^1)$ and $\mathbb{G}_2=(H_{2n}^2,
\sigma^2,\gamma^2)$ are isomorphic if there is an bijection, mapping $H_{2n}^1,
\sigma^1$ and $\gamma^1$ into $H_{2n}^2,\sigma^2$ and $\gamma^2$, respectively.
\end{definition}

By construction, this bijection preserves the cyclic order of the sectors around the
vertices as well as the order of the sectors along the boundary component and maps
ribbons into ribbons.
There are many additional ways to define isomorphisms of fatgraphs, see
\cite{Penner:10}. Definition~\ref{D:iso} is tailored to facilitate the
identification of components with irreducible fatgraphs, see
Section~\ref{S:irreducible}.

\begin{figure}[t]
\begin{center}
  \includegraphics[width=0.8\columnwidth]{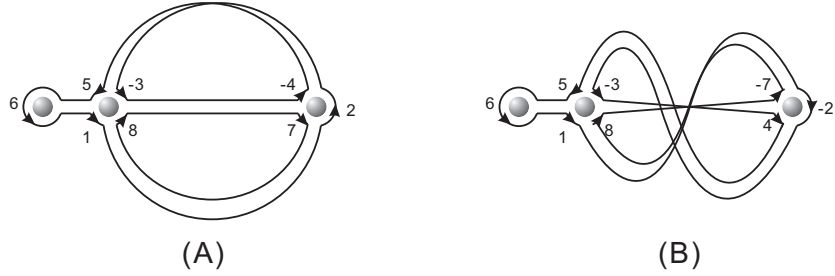}
\end{center}
\caption{\small
(A) a fatgraph with vertices $(1,8,-3,5),(6),(2,-4,7)$ and boundary components
$(1,2,-3,-4,5,6),(7,8)$. For the pair of sectors, $(-3,5)$, there is the corresponding
pair, $(2,-4)$, such that $-3=\gamma(2)$ and $5=\gamma(4)$, forming a twisted ribbon.
There are in addition the three untwisted ribbons: $((1,8),(7,2))$, $((5,1),(6,6))$ and
$((8,-3),(-4,7))$.
(B) flipping $(2,-4,7)$ to $(-2,-7,4)$: the flipping changes the twist property of ribbons
connecting $(2,-4,7)$ to other vertices.
}
\label{F:flip}
\end{figure}

A fatgraph, $\mathbb{G}$, represents a cell-complex of a surface $F(\mathbb{G})$:
the topological quotient space $F(\mathbb{G})$ is obtained by identifying the sides
of the $\mathbb{G}$-ribbons using the simplicial homeomorphism, see \cite{Penner:10}.
Accordingly the genus, $g$, of $\mathbb{G}$ is the topological genus of $F(\mathbb{G})$,
i.e.~
$$
2-g-b=v-e,
$$
where $b,e,v$ are the numbers of boundary components, ribbons and vertex-cycles of
$\mathbb{G}$. We shall write $\mathbb{G}_{n,g}$ if we wish to emphasize that $\mathbb{G}$
is a fatgraph having $2n$ sectors and genus $g$.

\begin{lemma}\label{L:dual}{\bf (Poincar\'e dual)}
Let $\mathbb{G}_{n,g}=(H_{2n},\sigma,\gamma)$ be a fatgraph, then we have:\\
{\rm (a)} $\mathbb{G}^*_{n,g}=(H_{2n},\gamma,\sigma)$ is a fatgraph and
          $(\mathbb{G}_{g,n}^*)^*=\mathbb{G}_{n,g}$,\\
{\rm (b)} if $\mathbb{G}_{n,g}=(H_{2n},\sigma,\gamma)$ is orientable then
          $\mathbb{G}^*_{n,g}=\mathbb{G}_{n,g}$.
\end{lemma}
\begin{proof}
Switching $\gamma$ and $\sigma$ for untwisted ribbons results in
$$
\diagram
x \dline^\gamma\rto^\sigma & \sigma(x) \\
\sigma(y)   & y\lto^\sigma\uline^\gamma
\enddiagram
\quad \mapsto \quad
\diagram
x \dline^\sigma\rto^\gamma & \gamma(x) \\
\gamma(y)   & y\lto^\gamma\uline^\sigma
\enddiagram
$$
Depending on the directions of the verticals we derive exactly one of the
following ribbons
$$
\diagram
\gamma(y) \rto^\sigma & x\dto^\gamma  \\
 y\uto^\gamma         & \gamma(x)\lto^\sigma
\enddiagram
\qquad
\diagram
\gamma(y) \rto^\sigma & x\dlto^\gamma \\
\gamma(x)            & y\lto^\sigma\ulto_\gamma
\enddiagram
\qquad
\diagram
x \rto^\sigma\dto^\gamma & \gamma(y)   \\
\gamma(x)          & y\lto^\sigma\uto^\gamma
\enddiagram
\qquad
\diagram
x \rto^\sigma\drto^\gamma & \gamma(y)   \\
y  \urto_\gamma                     & \gamma(x)\lto^\sigma
\enddiagram
$$
We analogously analyze the effect of switching $\gamma$ and $\sigma$ for
twisted ribbons:
$$
\diagram
x \drline^\gamma\rto^\sigma & \sigma(x) \\
\sigma(y)\urline_\gamma   & y\lto^\sigma
\enddiagram
\quad \mapsto\quad
\diagram
x \drline^\sigma\rto^\gamma & \gamma(x) \\
\gamma(y)\urline_\sigma   & y\lto^\gamma
\enddiagram
$$
where the verticals have an unique direction induced by the orientations
of the sectors. Depending on these directions we derive exactly one of the
following ribbons
$$
\diagram
x \dto^\gamma\rto^\sigma & y \dto^\gamma\\
   \gamma(x)       &   \lto^\sigma\gamma(y)
\enddiagram
\qquad
\diagram
x \drto_\gamma\rto^\sigma & y \dlto^\gamma\\
  \gamma(y)        &   \lto^\sigma\gamma(x)
\enddiagram
\qquad
\diagram
y \dto_\gamma\rto^\sigma & x \dto^\gamma\\
  \gamma(y)        &   \lto^\sigma\gamma(x)
\enddiagram
\qquad
\diagram
y \drto_\gamma\rto^\sigma & x \dlto^\gamma\\
  \gamma(x)        &   \lto^\sigma\gamma(y)
\enddiagram
$$
Accordingly, each $\mathbb{G}_{n,g}$-ribbon is mapped uniquely into a
ribbon in $\mathbb{G}^*_{n,g}$ and $\mathbb{G}^*_{n,g}=(H,\gamma,\sigma)$ is a
fatgraph. By construction this dualization preserves topological genus.

It remains to prove the second assertion. Suppose $\mathbb{G}_{n,g}=(H_{2n},\sigma,
\gamma)$ is orientable. Then all sectors are $\oplus$ and as a result we have only
untwisted ribbons of the following form
$$
\diagram
x \dto^\gamma\rto^\sigma & \sigma(x)=\gamma(y) \\
\sigma(y)=\gamma(x)   & y\lto^\sigma\uto^\gamma
\enddiagram
$$
Now switching $\sigma$ and $\gamma$ we derive
$$
\diagram
x \dto^\gamma\rto^\sigma & \gamma(y)=\sigma(x) \\
\gamma(x)=\sigma(y)   & y\lto^\sigma\uto^{\gamma}
\enddiagram
$$
Accordingly, each $\mathbb{G}_{n,g}$-ribbon is mapped identically to
a $\mathbb{G}^*_{n,g}$-ribbon.
Furthermore, if $\mathbb{G}_{n,g}=(H_{2n},\sigma,\gamma)$ is orientable, then
$\mathbb{G}^*_{n,g}=(H_{2n},\gamma,\sigma)$ is.
\end{proof}

\begin{figure}[t]
\begin{center}
\includegraphics[width=0.9\columnwidth]{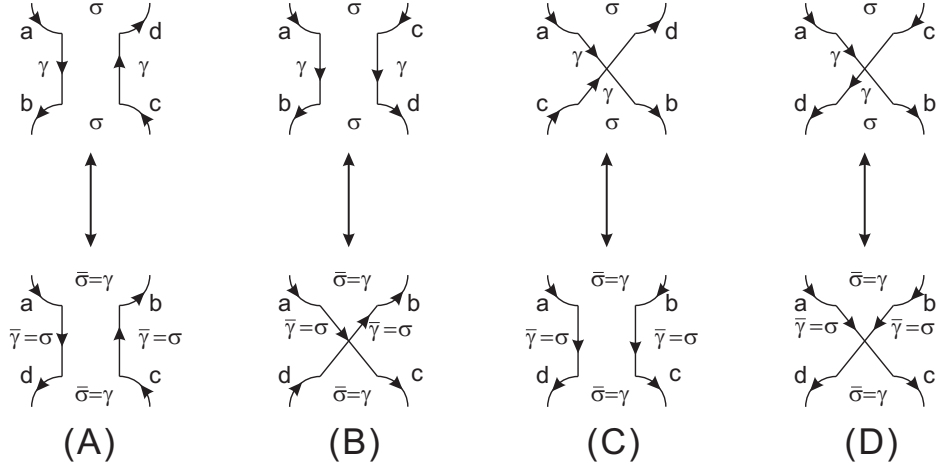}
\end{center}
\caption{\small Dualization: mapping ribbons into ribbons.
Assume that we have $b=\gamma(a)$ and $d=\gamma(c)$.
(A) $d=\sigma(a)$ and $b=\sigma(c)$,
(B) $c=\sigma(a)$ and $b=\sigma(d)$,
(C) $d=\sigma(a)$ and $c=\sigma(b)$,
(D) $c=\sigma(a)$ and $d=\sigma(b)$.
Let $\overline{\sigma}=\gamma$ and $\overline{\gamma}=\sigma$,
then $\overline{\sigma}(a)=\gamma(a)=b$ and $\overline{\sigma}(c)
=\gamma(a)=d$.
(A) $d=\overline{\gamma}(a)=\sigma(a)$
$b=\overline{\gamma}(c)=\sigma(c)$,
(B) $c=\overline{\gamma}(a)=\sigma(a)$ and
$b=\overline{\gamma}(d)=\sigma(d)$,
(C) $d=\overline{\gamma}(a)=\sigma(a)$ and
$c=\overline{\gamma}(b)=\sigma(b)$,
(D) $c=\overline{\gamma}(a)=\sigma(a)$ and
$d=\overline{\gamma}(b)=\sigma(b)$.
}
\label{F:ribbon}
\end{figure}


The Poincar\' e dual maps untwisted ribbons into $b$-ribbons and twisted ribbons
into $m$-ribbons. In general, the dual does not preserve twisted or untwisted ribbons
and in particular, the dual of a fatgraph having an unique vertex-cycle is a
fatgraph with an unique boundary component.
We illustrate dualization in
Fig.~\ref{F:fat}.

\begin{figure}[t]
\begin{center}
  \includegraphics[width=0.8\columnwidth]{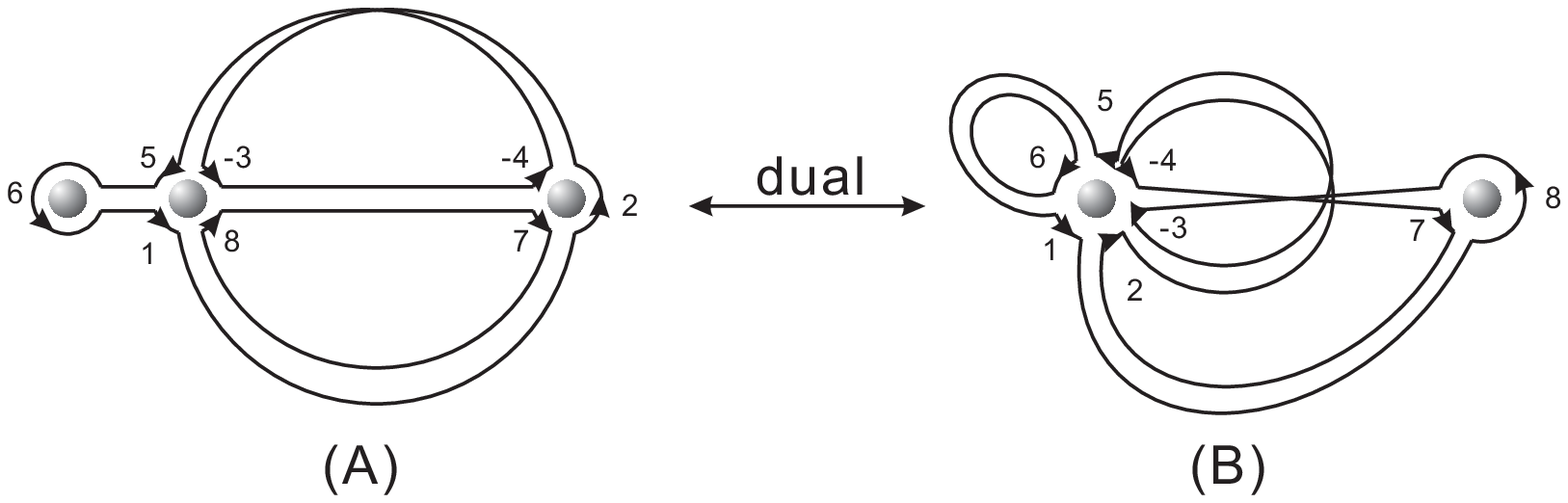}
\end{center}
\caption{\small
(A) a fatgraph with $\sigma=(1,8,-3,5)(6)(2,-4,7)$ and $\gamma=(1,2,-3,-4,5,6)(7,8)$.
(B) the dual of (A) obtained by exchanging $\sigma$ and $\gamma$. In general the
dualization does not the preserve twist property: the untwisted $m$-ribbon $((8,-3),(-4,7))$,
contained in (A), dualizes into the twisted $b$-ribbon $((-3,-4),(7,8))$.
}
\label{F:fat}
\end{figure}


\section{$\pi$-maps}\label{S:pi}

In this section, we shall give a bijection between signed permutations and
certain equivalence classes of unicellular fatgraphs. This correspondence
is the cornerstone of the paper.

Let $z_{2k+1}=(2k+1,\omega_{2k+1})$, i.e.~$z_{2k+1}$ has label $2k+1$ but arbitrary
orientation.

\begin{definition}\label{D:pi}{\bf ($\pi$-map)}
A unicellular fatgraph, $\mathbb{P}_{n,g}=(H_{2n+2},\gamma,\sigma)$,
is a $\pi$-map if it contains a vertex of the form
$$
v^*=(z_{2n+1},z_{2n-1},\ldots,z_{3},z_1).
$$
$V^*$ is called the center.
A $\pi$-map, $\mathbb{P}_{n,g}$, is called {\it reduced} if it does not contain
any vertices of degree one.
\end{definition}

In Fig.~\ref{F:id} we illustrate the concept of a $\pi$-map.

\begin{figure}[t]
\begin{center}
  \includegraphics[width=0.35\columnwidth]{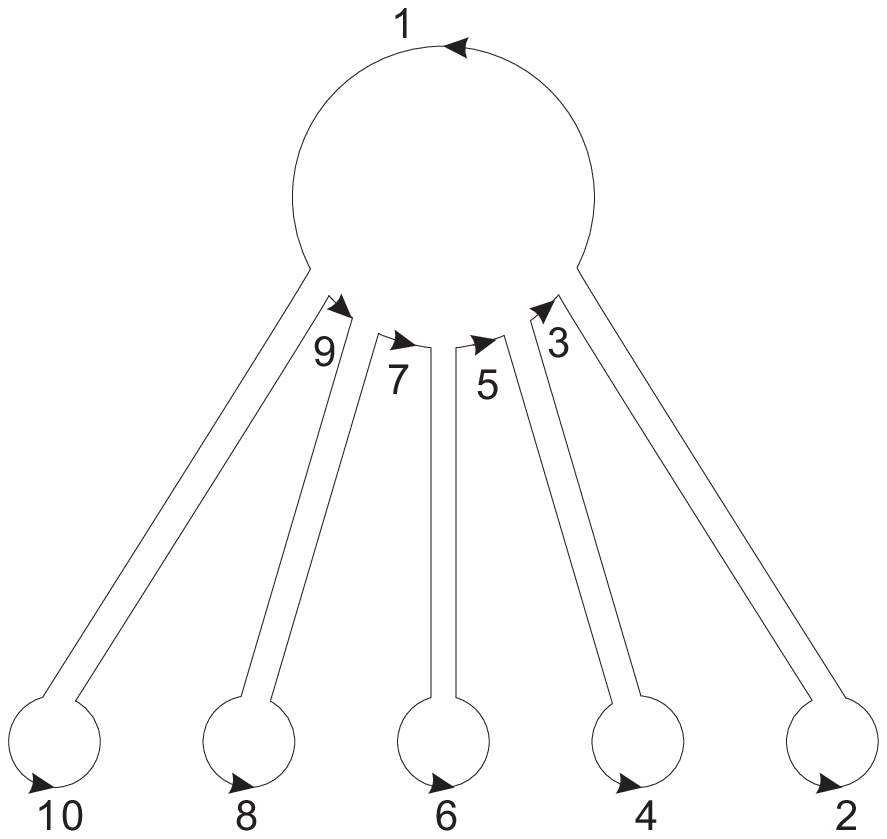}
\end{center}
\caption{\small
$\pi$-maps: here $v^*=(9,7,5,3,1)$ and
          $(2)(4)(6)(8)(10)$ are the external vertices.
}
\label{F:id}
\end{figure}


\begin{definition}\label{D:equivalence}{\bf (Equivalence)}
Two $\pi$-maps, $\mathbb{P}_{n,g}^1,\mathbb{P}_{n,g}^2$, with
$$\gamma^1 = (x_{2n+2}^1,x_{2n+1}^1,\ldots,x_2^1,x_1^1), \quad \text{and} \quad
\gamma^2 =
(x^2_{2n+2},x^2_{2n+1},\ldots,x_2^2,x_1^2)
$$
are equivalent, $\mathbb{P}^{1}_{n,g}
\sim \mathbb{P}^2_{n,g}$, iff
\begin{itemize}
\item $\mathbb{P}_{n,g}^1$ and $\mathbb{P}_{n,g}^2$ have the same center,
\item $\lambda_{x^1_{2k}}=\lambda_{x^2_{2\mu(k)}}$ for some permutation $\mu$,
\item the external $\mathbb{P}_{n,g}^1$-vertices can be transformed into
      $\mathbb{P}_{n,g}^2$-vertices via flipping.
\end{itemize}
Let $[\mathbb{P}_{n,g}]$ denote the equivalence class of $\mathbb{P}_{n,g}$ and
$\mathfrak{P}_{n,g}$ be the set of equivalence classes.
\end{definition}

Clearly, any two equivalent $\pi$-maps have the same topological genus. The equivalence class
is obtained by permutation of the even labels of the external vertices and by flipping them.

\begin{proposition}\label{P:char}
There exists a bijection $\varphi_n$ between the set of signed permutations,
$\mathfrak{B}_n$, and equivalence classes of $\pi$-maps, $\mathfrak{P}_{n,g}$,
$$
\varphi_n \colon \mathfrak{B}_n\longrightarrow \mathfrak{P}_n, \quad
b_n \mapsto [\mathbb{P}_{n,g}] .
$$
\end{proposition}
\begin{proof}
Given a signed permutation, $b_n=(\varepsilon_1y_1,\dots, \varepsilon_n y_n)$,
we associate to $\varepsilon_i y_i$ the sector $x_{2i+1}$, where $\lambda_{x_{2i+1}}=
2y_i+1$, and $\omega_{x_{2i+1}}=+$ if $\varepsilon_i=1$ and $-$, otherwise.
Let $x_1=(1,+)$ be an additional, $+$-sector, and
let
$$
\gamma^*=(z_{2n+1},z_{2n-1},\dots,z_5,z_3,z_1=x_1),
$$
with $z_{2k+1}=(2k+1,\omega_{z_{2k+1}})$.

{\it Claim.} Given $b_n$ and $x_{2k+1}$ constructed as above, there exists a set of fatgraphs
             $\mathbb{G}_{n,g}$ such that \\
             (a) $\gamma^*$ is a $\gamma$-cycle in $\mathbb{G}_{n,g}$ and
             $\sigma=(x_{2n+2},x_{2n+1},\dots,x_4,x_3,x_2,x_1)$, \\
             (b) any two of these, $\mathbb{G}_{n,g}$ and $\mathbb{G}_{n,g}'$ differ by choosing
             a labeling of the even sectors and an orientation of each boundary component
             over even sectors.

We shall interpret $\gamma^*=(z_{2n+1},z_{2n-1},\dots,z_5,z_3,z_1=x_1)$ as a boundary component
of length $n$ that traverses all sectors $z_{2n+1},\ldots,z_1$.

We make the Ansatz
$$
\sigma=(x_{2n+2},x_{2n+1},\dots,x_4,x_3,x_2,x_1),
$$
where $\lambda_{x_{2i}}= 2\mu(i)$ and $\mu\in S_n$, i.e.~the even sectors are arbitrarily
labeled with even numbers. Note that at this point the orientations of the even sectors
are not determined, yet. Furthermore, $\gamma^*$ can be expressed in terms
of the odd sectors $x_{2k-1}$. Then we have $\gamma^*=(x_{2\tau(i)-1})$, where
$x_{2\tau(i)-1}=z_{2i-1}$.

We proceed by producing the orientations of the even sectors $x_{2i}$ as well as the boundary
component $\gamma$, containing the cycle $\gamma^*$. These are constructed using the fact
that we have to generate ribbons.
This defines $\gamma$ via $\gamma(x_{2i})=x_{2j}$ as follows:
\begin{equation}\label{E:rib}
\diagram
x_{2i+1} \dto^{\gamma^*}\rto^\sigma   & x_{2i} \\
x_{2j-1}                           &  x_{2j} \lto^\sigma\uline^\gamma
\enddiagram
\quad \text{\rm or }\quad
\diagram
x_{2i+1} \drto^{\gamma^*}\rto^\sigma           & x_{2i} \\
x_{2j}   \urline_\gamma                      & x_{2j+1} \lto^\sigma
\enddiagram
\end{equation}
Accordingly the ribbon structure induces $\gamma$ as a collection of boundary components over
even sectors and the unique cycle of odd sectors, $\gamma^*$. The only choice is that of
selecting the orientation of one sector, $x_{2j}^{r}$, for each respective boundary component
over even sectors, $\gamma_r$.
The given orientations of the odd sectors naturally determine whether we have an untwisted
or a twisted ribbon and eq.~(\ref{E:rib}) defines $\gamma$ on even-indexed sectors, such
that $\sigma$ and $\gamma$ produce ribbons.

This construction is unique up to choosing an orientation in each $\gamma$-cycle,
$\gamma_r$, over even sectors. Eq.~(\ref{E:rib}) shows that this induces orientations for
all sectors (and the directions of the respective ribbon sides) contained in $\gamma_r$.
Accordingly, to a signed permutation corresponds the set of fatgraphs
$$
\mathbb{G}_{n,g}=(H_{2n+2},\sigma,\gamma),
$$
such that any two of them differ by choosing a labeling of the even sectors and
an orientation of each boundary component except $\gamma^*$. This proves the Claim.

We next consider the dual $\mathbb{P}_{n,g}=\mathbb{G}_{n,g}^*$. By construction,
in $\mathbb{P}_{n,g}$, $\gamma^*$ becomes the center
$$
v^*=(z_{2n+1},z_{2n-1},\ldots,z_5,z_3,z_1=x_1)
$$
and $\mathbb{P}_{n,g}$ is, by construction, unicellular having boundary component $\sigma$.
Note that in any $\pi$-map, $v^*$ contains the odd sectors labeled in descending order,
the only difference consists in their orientations.
Thus we have a well-defined mapping
$$
\varphi_n \colon \mathfrak{B}_n\longrightarrow \mathfrak{P}_{n,g}, \qquad
\varphi_n(b_n)=[\mathbb{P}_{n,g}].
$$
We proceed by constructing $\varphi_n^{-1}$: given an equivalence class of
$\pi$-maps $[\mathbb{P}_{n,g}]$, we choose a representant, $\mathbb{P}_{n,g}$ and
dualize. That is we choose $\mu$ and the orientations of the cycles over even sectors.
This produces the fatgraph $\mathbb{G}_{n,g}$ that has a boundary component cycle
$\gamma^*=(z_{2n+1},z_{2n-1},\ldots,z_5,z_3,z_1)$ traversing all odd sectors and the vertex
$$
\sigma=(x_{2n+2},x_{2n+1},\ldots ,x_4,x_3,x_2,x_1).
$$
To recover the signed permutation we only need partial information: the sequence
$$
(x_{2n+1},x_{2n-1},\ldots,x_5,x_3)
$$
is obviously independent of the choice of the representant. This induces an unique
signed permutation where the sign of $\varepsilon_i y_i$ equals the orientation of
the sector $x_{2i+1}$ and $y_i=(\lambda_{x_{2i+1}}-1)/2$.
\end{proof}

In Fig.~\ref{F:pi-map}, we give
an example to show how to construct a $\pi$-map from a signed permutation
and back.

\begin{figure}[t]
\begin{center}
\includegraphics[width=0.9\columnwidth]{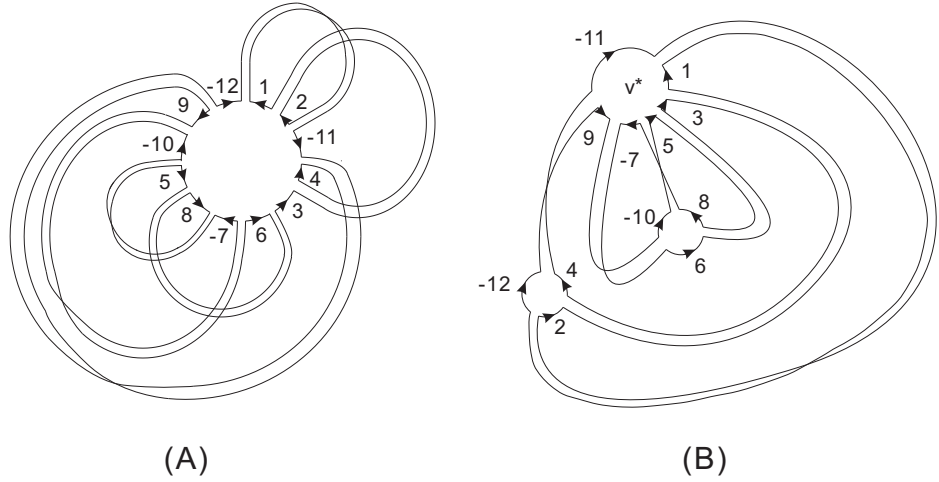}
\end{center}
\caption{\small From the
signed permutation $[-5,+1,-3,+2,+4]$ to its $\pi$-map: first we compute
$x_{11}=9$, $x_9=5$, $x_7=-7$, $x_5=3$,
$x_3=-11$ and $x_1=1$.
Then we set
$
\sigma=(x_{12}=-12, x_{11}=9,\ldots, x_3=-11, x_2=2, x_1=1),
$
where $2 \mu(i)=\lambda_{x_{2i}}$, $\mu\in S_n$ and
$x_{2i} =\sigma(x_{2i+1})$. Furthermore,
$
\gamma^*=(-11, 9, -7, 5, 3, 1).
$
This produces the fatgraph $\mathbb{G}_5$, (A).
In order to recover the signed permutation, we extract the sequence
$x_{11}=9$, $x_9=5$, $x_7=-7$, $x_5=3$, $x_3=-11$. Then the signed permutation
is $[-5,+1,-3,+2,+4]$ obtained via $y_i=(\lambda_{x_{2i+1}}-1)/2$, for
$1\le i \le 5$.
In (B) we depict the $\pi$-map of (A) after dualization.
}
\label{F:pi-map}
\end{figure}


We next collect some facts about $\pi$-maps.

\begin{lemma}\label{L:char2}
Let $b_n$ be a signed permutation and $[\mathbb{P}_{n,g}]=\varphi_n(b_n)$,
and $\mathbb{G}_{n,g}^*=\mathbb{P}_{n,g}$. Then \\
{\rm (a)} any $\mathbb{P}_{n,g}$-ribbon is incident to $v^*$, \\
{\rm (b)} a $\mathbb{P}_{n,g}$-ribbon is a $m$-ribbon if and only if it is induced by a
          twisted $\mathbb{G}_{n,g}$-ribbon.
\end{lemma}
\begin{proof}
(a) immediately follows, since $\mathbb{P}_{n,g}$ is the dual of $\mathbb{G}_{n,g}$
induced by $b_n$ and we have:
$$
\diagram
x_{2i-1} \dto^{\gamma^*}\rto^\sigma   & x_{2i} \\
x_{2j-1}                           &  x_{2j-2} \lto^\sigma\uline^\gamma
\enddiagram
\quad \mapsto\quad
\diagram
x_{2i-1} \dto^{\sigma}\rto^{\gamma^*}  & x_{2j-1} \\
x_{2i}                           &  x_{2j-2} \lline^\gamma\uto^\sigma
\enddiagram
$$
The case of twisted ribbons is analogous.

Ad (b): this follows from Lemma~\ref{L:dual}. Switching $\gamma$ and $\sigma$ for
twisted $\mathbb{G}_{n,g}$-ribbons (induced by the sign-change of the $\gamma^*$-sectors,
$x_{2i-1},x_{2j-1}$) means
$$
\diagram
x_{2i-1} \drline^\gamma\rto^\sigma & x_{2i} \\
x_{2j}\urline_\gamma   & x_{2j-1}\lto^\sigma
\enddiagram
\quad \mapsto\quad
\diagram
x_{2i-1} \drline^\sigma\rto^\gamma & x_{2i} \\
x_{2j}\urline_\sigma   & x_{2j-1}\lto^\gamma
\enddiagram
$$
The latter diagram is equivalent to exactly one of the following ribbons, depending on the
directions of the diagonals,
$$
\diagram
x_{2i-1} \dto^\gamma\rto^\sigma & x_{2j-1} \dto^\gamma\\
  x_{2i}       &   \lto^\sigma  x_{2j}
\enddiagram
\qquad
\diagram
x_{2i-1} \drto_\gamma\rto^\sigma & x_{2j-1} \dlto^\gamma\\
  x_{2j}        &   \lto^\sigma x_{2i}
\enddiagram
\qquad
\diagram
x_{2j-1} \dto_\gamma\rto^\sigma & x_{2i-1} \dto^\gamma\\
 x_{2j}        &   \lto^\sigma x_{2i}
\enddiagram
\qquad
\diagram
x_{2j-1} \drto_\gamma\rto^\sigma & x_{2i-1} \dlto^\gamma\\
  x_{2i}        &   \lto^\sigma x_{2j}
\enddiagram
$$
\end{proof}


We call a $\pi$-map, $\mathbb{P}_{n,g}$, orientable if it induces an orientable surface,
$F(\mathbb{P}_{n,g})$, and non-orientable, otherwise. Note that the notion of orientability
here is different from the notion of oriented cycles in break-point graphs \cite{Pevzner:99}.
In fact, a permutation with oriented cycles in a breakpoint graph corresponds one to one
to an non-orientable $\pi$-map.

\begin{lemma} \label{L:orientable}{\bf (Non-orientability)}
Let $\mathbb{P}_{n,g}$ be a reduced $\pi$-map. Then the following assertions are
equivalent:\\
{\rm (a)} $\mathbb{P}_{n,g}$ is non-orientable,\\
{\rm (b)} $\mathbb{P}_{n,g}$ contains a $m$-ribbon.\\
{\rm (c)} ${\mathbb{P}}_{n,g}$ contains an external vertex incident
          to both: a twisted as well as an untwisted ribbon.
\end{lemma}

\begin{proof}
{\rm (a)} $\Rightarrow$ (b):
suppose $\mathbb{P}_n$ contains no $m$-ribbon. Then $\mathbb{P}_n$ represents a
$2$-dimensional cell-complex with exclusively complementary edge pairs. By the
structure theorem of surfaces its underlying topological quotient space is an
orientable surface, whence (a) $\Rightarrow$ (b).

(b) $\Rightarrow$ (a): suppose $\mathbb{P}_n$ contains a $m$-ribbon, $e$.
Since any ${\mathbb{P}}_n$-ribbon is incident to $v^*$ we have a single vertex
containing two subsequent sectors with different orientations. Flipping external
vertices does not change the fact that $e$ is an $m$-ribbon and $v^*$ cannot be
flipped, by construction. This implies that the quotient space of $\mathbb{P}_n$
is a connected sum of projective planes and as such non-orientable.

(a) $\Rightarrow$ (c): follows by transposition.\\
(c) $\Rightarrow$ (a): the existence of such an external vertex implies that $v^*$
contains two subsequent sectors with different orientations and hence a $m$-ribbon.
We then employ (b) $\Rightarrow$ (a).
\end{proof}

Removing all vertices incident to exactly one ribbon and subsequent relabeling
induces a projection map $\mathbb{P}\mapsto \mathbb{P}^\rho$ from $\pi$-maps
to reduced $\pi$-maps. By Euler's characteristic equation this projection preserves
topological genus.

\begin{lemma}\label{L:reduced}
Let $b_n$ be a signed permutation and $[\mathbb{P}_{n,g}]=\varphi_n(b_n)$. Then $b_n$ and
the signed permutation $\varphi^{-1}([\mathbb{P}_{n,g}^\rho])$ have equal reversal
distance.
\end{lemma}
\begin{proof}
For an external vertex, that is incident to only one ribbon in $\mathbb{P}_n$
we have by Lemma~\ref{L:char2}, (a) the following alternative
$$
\diagram
x_{2i+1} \dto^{\sigma}\rto^{\gamma^*}  & x_{2j+1} \\
x_{2i}                           &    x_{2i}  \lline\uto^\sigma
\enddiagram
\quad \text{\rm or} \quad
\diagram
x_{2i+1} \rto^{\gamma^*}         & x_{2j+1} \dlto^\sigma\\
x_{2j}                          & x_{2j}\ulto^\sigma  \lline
\enddiagram
$$
where the orientations of $x_{2i+1},x_{2j+1}$ are equal, both being either $\oplus$
or $\ominus$.
Without loss of generality, we may assume that $x_{2i},x_{2j}$ is $\oplus$, i.e.~we have
$$
\lambda_{x_{2i+1}}+2=\lambda_{x_{2j+1}}.
$$
As a result we have, depending on the orientations of $x_{2i+1},x_{2j+1}$ being $\oplus$
or $\ominus$, either two successive numbers $(k-1),k$ or $-k,-(k-1)$, respectively, in
$\varphi_n^{-1}([\mathbb{P}_{n,g}])$.
Removing the sector $x_{2i}$ and replacing the sectors $x_{2i+1},x_{2j+1}$ by $x_{2i+1}$
is equivalent to replacing $\varphi_n^{-1}([\mathbb{P}_{n,g}])$ by the signed permutation
$b_{n-1}$, obtained by removing $k$ or $-k$, respectively and by setting $\pm j\mapsto
\pm (j-1)$ for any $j>k$, see Fig.~\ref{F:collapse}.
Clearly, $b_n$ and $b_{n-1}$ have the same reversal
distance, whence the lemma.
\end{proof}

\begin{corollary}\label{C:trivial}
The signed permutation corresponding to the equivalence class of $\pi$-maps having
exclusively external vertices of degree one, has reversal distance zero.
\end{corollary}

\begin{figure}[t]
\begin{center}
\includegraphics[width=0.7\columnwidth]{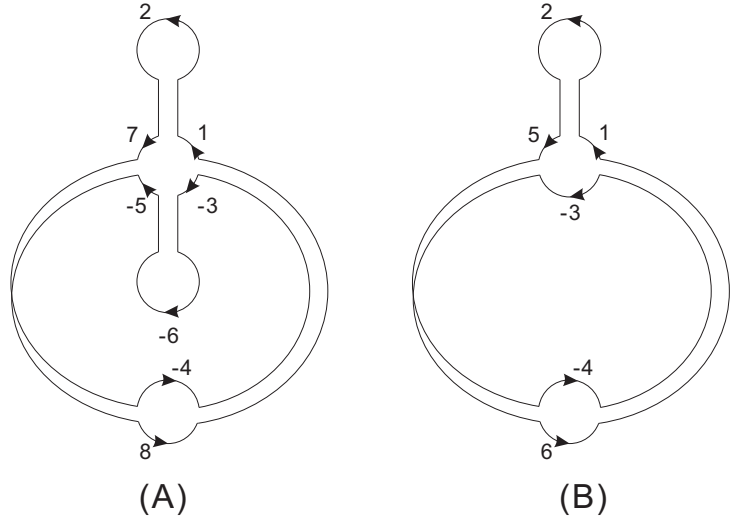}
\end{center}
\caption{\small (A) the $\pi$-map induced by the signed permutation $[-2, -1, 3]$.
(B) removing the vertex of degree one, $(-6)$, results in a $\pi$-map corresponding
to the signed permutation $[-1, 2]$. This is equivalent to replacing the segment
$[-2,-1]$ by $[-1]$. Both $[-2,-1,3]$ and $[-1,2]$ have reversal distance $1$.
}
\label{F:collapse}
\end{figure}


\section{Irreducibility and components}\label{S:irreducible}

Let $\mathbb{P}_{n,g}=(H_{2n+2},\sigma,\gamma)$ be a $\pi$-map with center
$v^*=(z_{2n+1},z_{2n-1},\ldots,z_1)$ where $\lambda_{z_{2k+1}}=2k+1$. By
construction, $v^*$ contains all odd $\mathbb{P}_{n,g}$-sectors.

By abuse of notation we shall identify the $v^*$-sectors with their labels,
when their respective orientation is not of relevance.
Let $i$ and $j$ be two $v^*$-sectors where $i<_\sigma j$, we set
\begin{eqnarray*}
[i,j]_\sigma & = &  \{t \mid t=2k+1, \ i \le_\sigma t \le_\sigma j\} \\
{[i,j]}_\gamma & = & \{ x \mid x=2k+1, \ i \le_\gamma x \le_\gamma j \ \vee
                                j \le_\gamma x \le_\gamma i \}.
\end{eqnarray*}
Note that in both intervals, $[a,b]_\sigma$ and $[a,b]_\gamma$, all sectors
are odd.
For any $[i,j]_\sigma$ the construction of Proposition~\ref{P:char} possibly
produces a $\pi$-map to which we refer to as $\mathbb{P}_{n,g}^{i,j}$.
An interval $[i,j]_\sigma$ inducing the $\pi$-map $\mathbb{P}_{n,g}^{i,j}$
is called minimal, if there exists no sector $i<_\sigma k<_\sigma j$ such that
both $[i,k]_\sigma$ and $[k,j]_\sigma$ induce $\pi$-maps, respectively.

Let $I_{\mathbb{P}_{n,g}}$ denote the set of intervals $[i,j]$ inducing $\pi$-maps,
$\mathbb{P}_{n,g}^{i,j}$. Then $I_{\mathbb{P}_{n,g}}$ becomes via inclusion, i.e.~
$$
[a,b]_\sigma\subset  [c,d]_\sigma \quad \Longleftrightarrow \quad
                                             c<_\sigma a <_\sigma b <_\sigma d,
$$
a partially ordered set.
By construction $[1,2n+1]_\sigma$ is the unique maximal element of
$(I_{\mathbb{P}_{n,g}},\subset)$.
\begin{definition}{\bf (Irreducibility)}
A $\pi$-map, $\mathbb{P}_{n,g}$, is called {\it irreducible} if
$I_{\mathbb{P}_{n,g}}=\{[1,2n+1]\}$.
\end{definition}

\begin{lemma}\label{L:gs}
Suppose $[i,j]_\sigma$ induces a $\pi$-map, then we have
$$
[i,j]_\sigma=[i,j]_\gamma.
$$
\end{lemma}
\begin{proof}
Let $t\in [i,j]_\sigma$, since $[i,j]_\sigma$ induces the $\pi$-map,
$\mathbb{P}^{i,j}_{n,g}$, $t$ is necessarily in its boundary component,
i.e.~$t\in [i,j]_\gamma$ and $[i,j]_\sigma\subseteq [i,j]_\gamma$.
$[i,j]_\gamma\subseteq [i,j]_\sigma$ follows analogously.
\end{proof}

\begin{definition}{\bf ($\mathbb{P}_n$-component)}
Let $\mathbb{P}_{n,g}$ be a $\pi$-map with minimal interval $[i,j]_\sigma$.
An interval, $[a_k, b_k]_\sigma\subsetneq [i,j]_\sigma$, is gap if it induces
a $\pi$-map and is of maximal length.
Suppose
$$
v^*=(\ldots, i,\ldots, a_k,a_ k',\ldots,b_k',b_k,\ldots,j\ldots),
$$
then the disjoint union of intervals
$$
C=[i,j]_\sigma\setminus \bigcup_k [a'_k,b'_k]_\sigma
$$
is called a $\mathbb{P}_n$-{component}.
A component $C=[i,i+2]_\sigma$ is called trivial.
\end{definition}

We illustrate the concept of a component in Fig.~\ref{F:component}.
\begin{figure}[t]
\begin{center}
  \includegraphics[width=0.5\columnwidth]{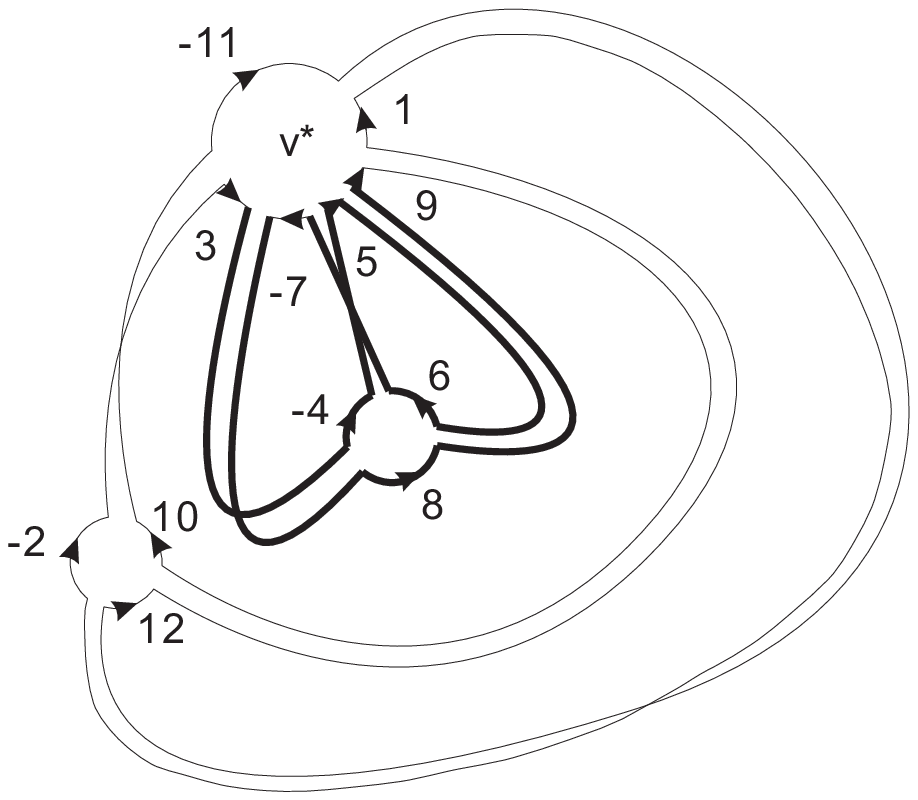}
\end{center}
\caption{\small Components:
a $\pi$-map consisting of two components $[3,9]_{\sigma}$ and
$[1,11]_{\sigma} \setminus [3,9]_{\sigma}$.
}
\label{F:component}
\end{figure}

Let $C$ be a $\mathbb{P}_{n,g}$-component. Collapsing all $C$-gaps and
subsequent relabeling of the sectors we derive:

\begin{lemma}\label{L:decomp}
Let $\mathbb{P}_{n,g}$ be a $\pi$-map, then the following assertions hold\\
{\rm (a)} any $\mathbb{P}_{n,g}$-component is isomorphic to an irreducible
$\pi$-map,\\
{\rm (b)} any $\pi$-map can be uniquely decomposed into a set of components.
\end{lemma}

A component is orientable if it is nontrivial and its associated, irreducible
$\pi$-map is orientable, i.e.~by Lemma~\ref{L:orientable}, it contains only
$b$-ribbons.

As for the relation between two distinct component $C_1,C_2$ we observe
\begin{lemma}\label{L:location}
Let $C_1 = \bigcup_j^n [c_j^{(1)}, d_j^{(1)}]_\sigma$ and $C_2 =
\bigcup_j^m [c_i^{(2)}, d_i^{(2)}]_\sigma$
be two components. Then
$$
\exists \ 1\le j< n-1; \  d_j^{(1)} \le_\sigma c_1^{(2)}
\quad \Longrightarrow \quad  C_2\subset [d_j^{(1)},c_{j+1}^{(1)}]_\sigma.
$$
That is, two components are either subsequent around $v^*$ or one is
contained in a gap of the other.
\end{lemma}
\begin{proof}
By construction, the intervals of two components only intersect on their boundaries.
Suppose $d_j^{(1)} \le_\sigma c_1^{(2)}$ for some $1\le j< n-1$.
By assumption $[d_j^{(1)},c_{j+1}^{(1)}]_\sigma$ is a gap and as such it induces a
maximal $\pi$-map. Since $c_1^{(2)}\in [d_j^{(1)},c_{j+1}^{(1)}]_\sigma$ this implies
$C_2\subset [d_j^{(1)},c_{j+1}^{(1)}]_\sigma$.
\end{proof}

By Lemma~\ref{L:char2} any $\mathbb{P}_{n,g}$-ribbon, $t$, is incident to $v^*$
i.e.~it is determined by its pair of incident, odd sectors $(t,\sigma(t))$.

\begin{definition}{\bf ($\sigma$-crossing)}
Let $v_1,v_2$ be two external vertices and let
$$
[v]=\{r\mid \text{\rm $r$ is a ribbon that is incident to $v$}\}.
$$
Then $v_1,v_2$ are $\sigma$-crossing
if and only if there exist four ribbons $r_1,r_3\in [v_1]$, $r_2,r_4\in [v_2]$
such that
$$
t_{r_1}<_\sigma t_{r_2} \le_\sigma t_{r_3} <_\sigma t_{r_4}.
$$
\end{definition}
Let $(\bullet)$ be the following property:
for any for two external vertices $v_1,v_2$ there exists a sequence
$(v_1=w_1,w_2,\ldots,w_{k-1},w_k=v_2)$ such that $w_i,w_{i+1}$ are
$\sigma$-crossing, see Fig.~\ref{F:scross}.

\begin{figure}[t]
\begin{center}
  \includegraphics[width=0.5\columnwidth]{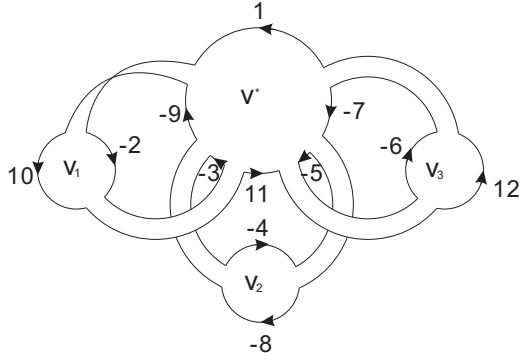}
\end{center}
\caption{\small $\sigma$-crossing: $v_1$ and $v_2$ are $\sigma$ crossing since
$-9<_\sigma -3\le_\sigma -3 <_\sigma 11$, where ribbon $(-9, -3)\in [v_2]$ and ribbon
$(-3, 11)\in [v_1]$.
$v_1$ and $v_3$ are not $\sigma$-crossing.
}
\label{F:scross}
\end{figure}


\begin{lemma}\label{L:equivalence}
Let $\mathbb{P}_{n,g}$ be a $\pi$-map, then the following assertions are equivalent;\\
{\rm (a)} $\mathbb{P}_{n,g}$ is irreducible,\\
{\rm (b)} $\mathbb{P}_{n,g}$ satisfies $(\bullet)$.
\end{lemma}
\begin{proof}
${\rm (a)} \Rightarrow {\rm (b)}$: for an arbitrary external vertex, $v$, let
$S(v)$ be the set of $v^*$-sectors associated to ribbons that are either
incident to $v$ or incident to vertices crossing $v$. The sets $S(v)$ are
partially ordered via
$$
S(v)\le S(v')\quad  \Longleftrightarrow \quad
{[\min S(v),\max S(v)]_\sigma\subset [\min S(v'),\max S(v')]_\sigma} .
$$
Let $S(v_0)$ be a $\le$-minimal element. Then $S(v_0)$ is an interval $[i,k]_\sigma$
such that any external vertex incident to a ribbon attached to $[i,k]_\sigma$
has all its incident ribbons attached to $[i,k]_\sigma$. This means $[i,k]_\sigma$
forms an induced $\pi$-map. Since $\mathbb{P}_{n,g}$ is irreducible we have
$[i,k]_\sigma=[1,2n+1]_\sigma$, whence for any two external vertices $(\bullet)$
holds.

${\rm (b)} \Rightarrow {\rm (a)}$: suppose $\mathbb{P}_{n,g}$ is not irreducible.
Then there exists some interval $[i,k]_\sigma\neq [1,2n+1]_\sigma$ such that
$\mathbb{P}_{n,g}^{i,k}$ is a $\pi$-map.
Let $v_1$ be an external vertex that is incident to a ribbon attached to $[i,k]_\sigma$
and let $v_2$ be an external vertex, incident to a ribbon not attached to $[i,k]_\sigma$.

By Lemma~\ref{L:gs}, since $\mathbb{P}_n^{i,k}$ is a $\pi$-map, we have
$[i,j]_\sigma=[i,j]_\gamma$ and any vertex incident to a $[i,k]_\sigma$-ribbon
has all its incident ribbons attached to $[i,k]_\sigma$.
Consequently, for $v_1$ and $v_2$ a sequence $(v_1=w_1,w_2,\ldots,w_{k-1},w_k=v_2)$
of mutually crossing vertices does not exist. This means that $(\bullet)$ implies
that $\mathbb{P}_{n,g}$ is irreducible.
\end{proof}

\section{Reversals}\label{S:reversal}

In this section we study the action of reversals on $\pi$-maps. Suppose the signed permutation
$$
b_n=(\epsilon_1 y_1, \ldots, \epsilon_1 y_i, \ldots, \epsilon_j y_j,\ldots,\epsilon_n y_n),
$$
is acted upon by the reversal $\rho_{i,j}$. The action produces the signed permutation
$$
b_n\cdot \rho_{i,j} =\overline{b}_n = (\epsilon_1 y_1, \ldots, -\epsilon_j y_j, \ldots,
-\epsilon_i y_i,\ldots,\epsilon_n y_n).
$$
$b_n$ and $\overline{b}_n$ induce by Proposition~\ref{P:char} the equivalence classes
$[\mathbb{P}_n]$ and $[\overline{\mathbb{P}}_n]$. That is we have the diagram
\begin{equation}\label{E:must}
\diagram
b_n  \dto^{\rho_{i,j}}\rrto^{\varphi_n}  & & [\mathbb{P}_n] \dto^{\rho_{i,j}}\\
\overline{b}_n \rrto^{\varphi_n}       & & [\overline{\mathbb{P}}_n]
\enddiagram
\end{equation}
Accordingly, we have a natural reversal action on equivalence classes of $\pi$-maps induced by
the reversal-right multiplication in the group of signed permutations and making the
above diagram commutative.

In order to describe the action of reversals on $[\mathbb{P}_n]$ combinatorially,
we reconsider the relation between sectors and ribbons in fatgraphs. On the one
hand a ribbon is a diagram
$$
\diagram
x \dline^\gamma\rto^\sigma & \sigma(x) \\
\sigma(y)   & y\lto^\sigma\uline^\gamma
\enddiagram
\quad \text{\rm or }\quad
\diagram
x \drline^\gamma\rto^\sigma & \sigma(x) \\
\sigma(y)\urline_\gamma   & y\lto^\sigma
\enddiagram
$$
where the directions of the verticals labeled by $\gamma$ are implied by the orientations
of $x,\sigma(x), y,\sigma(y)$ determined by these four sectors. On the other hand, a sector
$\sigma(x)$ is determined by the pair of its incident ribbons
$$
\diagram
x \dline^\gamma\rto^\sigma & \sigma(x) \\
\sigma^2(y)   & \sigma(y) \lto^\sigma\uline^\gamma
\enddiagram
\quad
\diagram
\sigma(x) \dline^\gamma\rto^\sigma & \sigma^2(x) \\
\sigma(y)   & y\lto^\sigma\uline^\gamma
\enddiagram
$$
depicted here, w.l.o.g.~as being untwisted. Furthermore, any even vertex can
be described by the sequence of its incident ribbons
$(e_1, \ldots, e_k)$, such that $e_i=(\sigma^{-1}(t),t)$ and
$e_{i+1}=(t, \sigma(t))$, for $1\le i\le k$.

Let $\mathbb{P}_n$ be a $\pi$-map with boundary component
$$
\gamma=(x_{2n+2},\ldots,x_{2j+2},x_{2j+1},\ldots,x_{2i+1},x_{2i},x_{2i-1},\ldots,x_1).
$$

{\bf Scenario $1$ (gluing):}
suppose the sectors $x_{2j+2}$ and $x_{2i}$ are located at the two distinct vertices
$v_1$ and $v_2$.
Without changing the equivalence
class of $\mathbb{P}_n$, we can, (by means of flipping $v_2$ if necessary, which has no
effect on the $v^*$-sectors)
assume that $v_1$ and $v_2$ are given by
\begin{equation}\label{E:ccc}
v_1=(\underbrace{(\lambda_{x_{2j+2}}, +)}_{x_{2j+2}},h_1^{(1)}, \ldots, h_{m_1}^{(1)}), \quad
\quad  v_2=(\underbrace{(\lambda_{x_{2i}}, -)}_{x_{2i}}, h_1^{(2)}, \ldots, h_{m_2}^{(2)}).
\end{equation}
We represent $v_1$ and $v_2$ by the sequences of ribbons
\begin{eqnarray*}
v_1 & = &  (\underbrace{(h_{m_1}^{(1)}, x_{2j+2})}_{e_1^{(1)}},
       \underbrace{(x_{2j+2}, h_1^{(1)})}_{e_2^{(1)}}, e_3^{(1)},\ldots, e_{m_1}^{(1)}), \\
v_2 & = & (\underbrace{(h_{m_2}^{(2)}, x_{2i})}_{e_1^{(2)}},
       \underbrace{(x_{2i}, h_{1}^{(2)})}_{e_2^{(2)}},e_3^{(2)}, \ldots, e_{m_2}^{(2)}).
\end{eqnarray*}
Next we set
$$
\tilde{v} =
(\underbrace{(h_{m_1}^{(1)}, x_{2j+2})}_{e_1^{(1)}}, \underbrace{(x_{2i}, h_{1}^{(2)})}_{e_2^{(2)}},
              e_3^{(2)},\ldots, e_{m_2}^{(2)}, \underbrace{(h_{m_2}^{(2)}, x_{2i})}_{e_1^{(2)}},
              \underbrace{(x_{2j+2}, h_1^{(1)})}_{e_2^{(1)}}, e_3^{(1)}, \ldots, e_{m_1}^{(1)})
$$
and obtain by replacing $v_1$ and $v_2$ by $\tilde{v}$ the fatgraph $\tilde{\mathbb{P}}_n$, see
Fig.~\ref{F:reversal_glue}.

\begin{figure}[t]
\begin{center}
  \includegraphics[width=0.8\columnwidth]{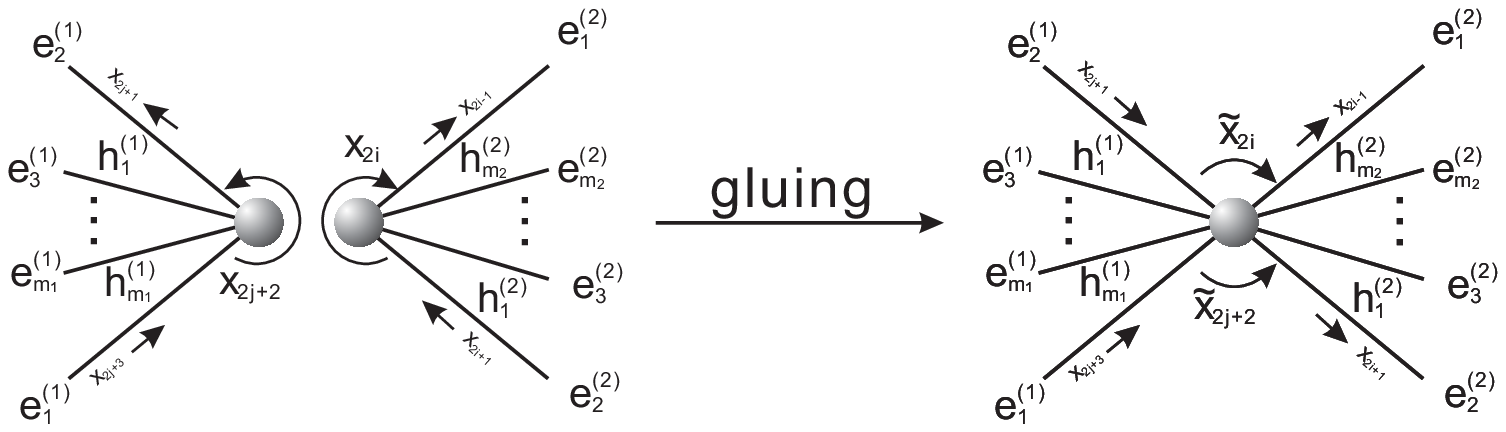}
\end{center}
\caption{\small Gluing.
}
\label{F:reversal_glue}
\end{figure}


{\bf Scenario $2$ (slicing):} suppose $x_{2j+2}$ and $x_{2i}$ are located at $v$ and that
$x_{2j+2}$ and $x_{2i}$ have different orientations.
Without changing the equivalence class of $\mathbb{P}_n$
we may assume that
$$
v=(\underbrace{(\lambda_{x_{2j+2}},+)}_{x_{2j+2}}, h_1^{(2)}, \ldots, h_{m_2}^{(2)},
   \underbrace{(\lambda_{x_{2i}},-)}_{x_{2i}}, h_1^{(1)}, \ldots, h_{m_1}^{(1)}).
$$
We next express $v$ as the sequence of ribbons
$$
v=(\underbrace{(h_{m_1}^{(1)}, x_{2j+2})}_{e_1^{(1)}},
\underbrace{(x_{2j+2}, h_1^{(2)})}_{e_2^{(2)}}, e_3^{(2)}, \ldots, e_{m_2}^{(2)},
\underbrace{(h_{m_2}^{(2)}, x_{2i})}_{e_1^{(2)}},
\underbrace{(x_{2i}, h_1^{(1)})}_{e_2^{(1)}},
 e_3^{(1)}, \ldots, e_{m_1}^{(1)}).
$$
Let
$$
\tilde{v}_1=(\underbrace{(h_{m_1}^{(1)}, x_{2j+2})}_{e_1^{(1)}},
\underbrace{(x_{2i}, h_1^{(1)})}_{e_2^{(1)}}, e_3^{(1)}, \ldots, e_{m_1}^{(1)}),  \quad
\tilde{v}_2=(\underbrace{(h_{m_2}^{(2)}, x_{2i})}_{e_1^{(2)}},
\underbrace{(x_{2j+2}, h_1^{(2)})}_{e_2^{(2)}}, e_3^{(2)}, \ldots, e_{m_2}^{(2)}).
$$
Replacing $v$ by $\tilde{v}_1$ and $\tilde{v}_2$ in $\mathbb{P}_n$, we obtain a fatgraph
$\tilde{\mathbb{P}}_n$ having the same ribbons as $\mathbb{P}_n$, see
Fig.~\ref{F:reversal_slice}.

\begin{figure}[t]
\begin{center}
  \includegraphics[width=0.8\columnwidth]{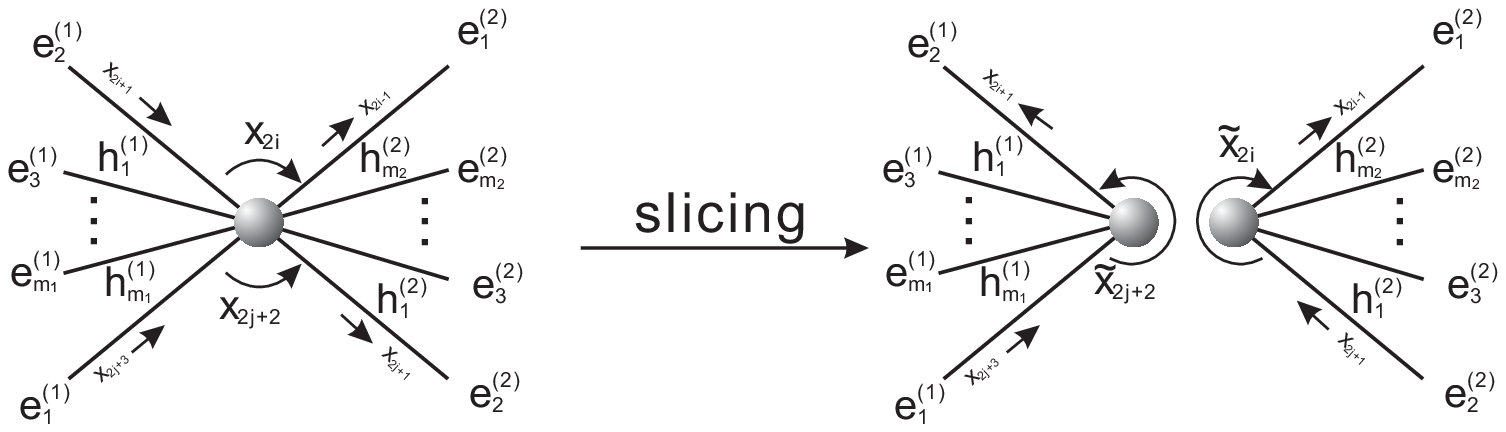}
\end{center}
\caption{\small
Slicing}
\label{F:reversal_slice}
\end{figure}


{\bf Scenario $3$ (half-flipping):}
suppose $x_{2j+2}$ and $x_{2i}$ are located at $v$ and suppose furthermore that
$x_{2j+2}$ and $x_{2i}$ have the same orientation, i.e.,
$$
v=(\underbrace{(\lambda_{x_{2j+2}},+)}_{x_{2j+2}}, h_1^{(1)}, \ldots, h_{m_1}^{(1)},
\underbrace{(\lambda_{x_{2i}},+)}_{x_{2i}}, h_1^{(2)}, \ldots, h_{m_2}^{(2)}).
$$
We represent $v$ as
$$
v=(\underbrace{(h_{m_2}^{(2)}, x_{2j+2})}_{e_1^{(1)}},
\underbrace{(x_{2j+2},  h_1^{(1)})}_{e_2^{(1)}}, e_3^{(1)}, \ldots, e_{m_1}^{(1)},
\underbrace{(h_{m_1}^{(1)}, x_{2i})}_{e_1^{(2)}},
\underbrace{(x_{2i},  h_1^{(2)})}_{e_2^{(2)}}, e_3^{(2)}, \ldots,
e_{m_2}^{(2)}).
$$
Next we set
$$
\tilde{v}=(\underbrace{(h_{m_2}^{(2)}, x_{2j+2})}_{e_1^{(1)}},
\underbrace{(x_{2i}, h_{m_1}^{(1)})}_{\overline{e}_1^{(2)}},
\overline{e}_{m_1}^{(1)}, \ldots, \overline{e}_3^{(1)},
\underbrace{(h_1^{(1)}, x_{2j+2})}_{\overline{e}_2^{(1)}},
\underbrace{(x_{2i}, h_1^{(2)})}_{e_2^{(2)}},
e_3^{(2)}, \ldots, e_{m_2}^{(2)}).
$$
Replacing $v$ by $\tilde{v}$ we obtain the fatgraph $\tilde{\mathbb{P}}_n$, see
Fig.~\ref{F:reversal_flip}.

\begin{figure}[t]
\begin{center}
  \includegraphics[width=0.8\columnwidth]{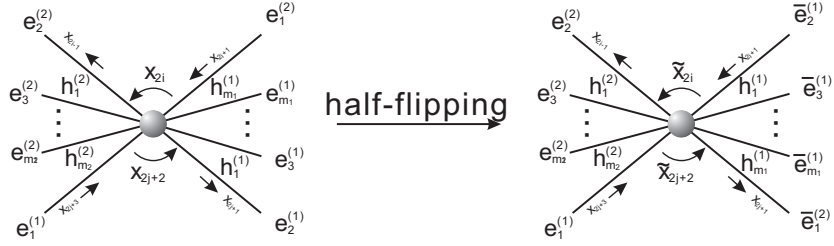}
\end{center}
\caption{\small Half-flipping
}
\label{F:reversal_flip}
\end{figure}


\begin{lemma}\label{L:model}
Let $\mathbb{P}_n$ be a $\pi$-map satisfying eq.~(\ref{E:ccc}) and having the boundary
component
$$
\gamma=(x_{2n+2},\ldots,x_{2j+2},x_{2j+1},\ldots,x_{2i+1},x_{2i},x_{2i-1},\ldots,x_1).
$$
Then $\tilde{\mathbb{P}}_n$ is a $\pi$-map having the boundary component
$$
\tilde{\gamma}= ({x}_{2n+2}, \ldots, \tilde{x}_{2j+2},
-{x}_{2i+1},\ldots,-{x}_{2i+2},\ldots, -x_{2j+1},\tilde{x}_{2i},\ldots,{x}_1),
$$
where $-x_k=(\lambda_{x_k},-\omega_{x_k})$ denotes the sector $x_k$ having reversed orientation.
\end{lemma}
\begin{proof}
By construction $\tilde{\mathbb{P}}_n$ is a fatgraph having a center vertex $\tilde{v}^*$.
As for $\tilde{\mathbb{P}}_n$-boundary components, we consider the sector $x_{2n+2}$.
The $\tilde{\mathbb{P}}_n$-boundary component starting at $x_{2n+2}$ visits
the same sectors and ribbons as in $\mathbb{P}_n$ before arriving at $\tilde{x}_{2j+2}$.
$\tilde{x}_{2j+2}$ is the sector formed by the pair of ribbons
\begin{eqnarray*}
{(h_{m_1}^{(1)}, x_{2j+2})}={e_1^{(1)}} &&
{(x_{2i}, h_{1}^{(2)})}={e_2^{(2)}} \quad \text{\rm (gluing)} \\
{(h_{m_1}^{(1)}, x_{2j+2})}={e_1^{(1)}} &&
{(x_{2i}, h_1^{(1)})}={e_2^{(1)}} \quad \text{\rm (slicing)} \\
{(h_{m_2}^{(2)}, x_{2j+2})}={e_1^{(1)}} &&
{(x_{2i}, h_{m_1}^{(1)})}={\overline{e}_1^{(2)}}\quad \text{\rm (half-flipping)}
\end{eqnarray*}
In all three cases, the next sector of the tour is traversing $x_{2i+1}$ in reverse
orientation, i.e.~we have $-{x}_{2i+1}=(\lambda_{x_{2i+1}},-\omega_{x_{2i+1}})$.
We continue now traversing the $\tilde{\mathbb{P}}_n$-boundary
component as the $\mathbb{P}_n$-boundary component, $\gamma$, in reverse
order via $-{x}_{2i+2},\ldots,-x_{2j+1}$. We then arrive at the new sector
$\tilde{x}_{2i}$, given by the pair of ribbons
\begin{eqnarray*}
{(h_{m_2}^{(2)}, x_{2i})}={e_1^{(2)}} & &
{(x_{2j+2}, h_1^{(1)})}={e_2^{(1)}} \quad \text{\rm (gluing)}\\
{(x_{2i}, h_1^{(1)})}={e_2^{(1)}} &&  {(x_{2j+2}, h_1^{(2)})}={e_2^{(2)}}
\quad \text{\rm (slicing)}\\
{(h_1^{(1)}, x_{2j+2})}={\overline{e}_2^{(1)}} & &
{(x_{2i}, h_1^{(2)})}={e_2^{(2)}} \quad \text{\rm (half-flipping)}
\end{eqnarray*}
The next sectors traversed by the $\tilde{\mathbb{P}}_n$-boundary component
are
$$
{x}_{2i-1}=x_{2i-1},x_{2i-2},\ldots,x_1
$$
and thus traversed as in the $\mathbb{P}_n$-boundary component $\gamma$.
Accordingly, $\tilde{\mathbb{P}}_n$ has the unique boundary component
$$
\tilde{\gamma}= ({x}_{2n+2}, \ldots, \tilde{x}_{2j+2}, -{x}_{2i+1}, \ldots, -x_{2i+2},
-{x}_{2j+1}, \tilde{x}_{2i}, \ldots, {x}_1)
$$
and the lemma follows.
\end{proof}

Inspecting the effect of Lemma~\ref{L:model} on the underlying signed permutation,
we shall categorize the action of reversals as to either glue, splice or
half-flip.

\begin{lemma}\label{L:3}
Let $\rho_{i,j}$ be a reversal acting on the $\pi$-map $[\mathbb{P}_{n,g}]$
and let $b_n=\varphi_n^{-1}([\mathbb{P}_{n,g}])$. Then
\begin{equation}
[\mathbb{P}_{n,g}]\cdot \rho_{i,j} = [\tilde{\mathbb{P}}_{n,g'}].
\end{equation}
That is $\tilde{\mathbb{P}}_{n,g'}\sim \overline{\mathbb{P}}_{n,g'}$, where
$\tilde{\mathbb{P}}_{n,g'}$ is obtained by either gluing, splicing or half-flipping.
Furthermore we have $\mid g-g'\mid\le 1$ and
$$
d(b_n)\ge g.
$$
\end{lemma}
\begin{proof}
The reversal $\rho_{i,j}$ determines uniquely the pair of even sectors $(x_{2i},x_{2j+2})$ in
$\mathbb{P}_{n,g}$. If the two sectors belong to two distinct $\mathbb{P}_{n,g}$-vertices, by
Proposition~\ref{P:char}, the orientations of $x_{2i},x_{2j+2}$ can be chosen as to
satisfy eq.~(\ref{E:ccc}).
Otherwise, $x_{2i},x_{2j+2}$ have either distinct or equal orientations.
This corresponds to the three scenarios: gluing, slicing and half-flipping.

By Lemma~\ref{L:model}, any of these generates the $\pi$-map $\tilde{\mathbb{P}}_{n,g'}$
having the boundary component $\tilde{\gamma}$, respectively. By Proposition~\ref{P:char}
a $\pi$-map with boundary component $\tilde{\gamma}$ induces an equivalence class
that corresponds to the signed permutation $b_n\rho_{i,j}$. Consequently we have
$$
\varphi_n^{-1}([\tilde{\mathbb{P}}_{n,g'}])=b_n\rho_{i,j}.
$$
Since $\varphi_n$ is a bijection between the set of signed permutations and equivalence
classes of $\pi$-maps we derive
$$
\tilde{\mathbb{P}}_{n,g'}\sim \overline{\mathbb{P}}_{n,g'}.
$$
Euler's characteristic equation immediately implies $\mid g-g'\mid\le 1$, whence
any reversal decreases the genus of the underlying $\pi$-map by at most one, i.e.
$$
d(b_n)\ge g.
$$
\end{proof}


\section{Non-orientable components}\label{S:genus}


By Lemma~\ref{L:decomp} any $\pi$-map $\mathbb{P}_{n,g}$ can be uniquely decomposed
into components, $C_i$, having genus $g_i$. Each of these is isomorphic to an
irreducible $\pi$-map and we have $\sum_i g_i=g$.

In the following we shall show that any non-orientable component, $C_i$, of
$\mathbb{P}_{n,g}$, or equivalently any non-orientable, irreducible $\pi$-map
can be spliced into $\varphi_n^{-1}([\text{\rm id}])$ using $g_i$ reversals.
This implies in particular the sharpness of the lower bound on the reversal
distance given by Lemma~\ref{L:3}.

The following result is due to \cite{Pevzner:99}. The proof given here is
based on the characterization of components via $\sigma$-crossings.

In the following we present Theorem 4 of \cite{Pevzner:99} employing the
$\pi$-map framework.
This theorem plays a key role of computing the reversal distance of signed
permutations. An irreducible, non-orientable $\pi$-map corresponds to an
oriented component in the breakpoint graph \cite{Pevzner:99}.

\begin{lemma}\cite{Pevzner:99}\label{L:splicedown}
Let $\mathbb{P}_{n,g}$ be an irreducible, non-orientable $\pi$-map of genus $g$,
then its associated, signed permutation, $b_n=\varphi_n^{-1}([\mathbb{P}_{n,g}])$,
has reversal distance $g$.
\end{lemma}

\begin{proof}
Since $C$ is non-orientable there exists some $m$-ribbon $e$. We shall show

{\it Claim $0$.}
There exists some $m$-ribbon $e$ in $C$ such that splicing $e$ decomposes $C$ into
exclusively non-orientable components. In particular, $C$ can be
successively spliced into trivial components.

By Lemma~\ref{L:orientable},
there exists some $m$-ribbon, $e$ in $C$. Suppose splicing $e$ decomposes $C$
into the components $C_1^e,\ldots,C_k^e$. If $C_1^e$ is orientable, then it contains
exclusively $b$-ribbons $y_1^{e}$. Since $C_1=\{y\in C\mid y^{e}\in C_1^{e}\}$ is,
due to the presence of $e$, not a component in $C$, not all of them can be $b$-ribbons.
That is, there exists some $b$-ribbon, $e_1^e$, that was originally a $m$-ribbon in $C$.
Since $e_1^e\in C_1^e$ and $C_1^{e}$ is orientable, $e_1^e$ is untwisted. Since splicing
does not change untwisted edges into twisted ones, we can conclude that $e_1$ is
untwisted.


{\it Claim $1$.} Splicing $e_1$ produces from $C$ a non-orientable
                 component $C_*^{e_1}$.

By Lemma~\ref{L:orientable}, $C$ contains an external vertex, $v$, incident
to a twisted and an untwisted ribbon, respectively. We shall show that also $v^{e_1}$,
the vertex obtained by slicing $e_1$, has this property in $C_*^{e_1}$.
Since $e_1$ is untwisted, we have to assure that splicing $e_1$ does not eliminate
the only untwisted $v$-ribbon.

Clearly, if $v$ is not incident $e_1$, then $v^{e_1}$ is still incident to a twisted
and an untwisted ribbon.

Otherwise, we observe that there exists at least one additional untwisted ribbon
incident to $v$. Indeed, splicing $e$ produces by assumption the orientable
component $C_1^e$, containing $e_1$. Suppose now $e_1$ were the only untwisted
ribbon incident to $v$. Then the orientability of $C_1^e$ guarantees that (a) $v^e$
cannot be incident to any twisted ribbon, since slicing $e$ does preserve twisted
and untwisted ribbons, (b) $C_1^e$ contains only the ribbon $e_1^e$. (a) and (b)
imply that $C_1^e$ is trivial, which is impossible and Claim $1$ follows.

We next show that $C_*^{e_1}$ is quite ``large'': let $R^e$ and $R^{e_1}$ denote the
sets of ribbons derived from $C$ by splicing $e$ and $e_1$, respectively. We call
two ribbons $y^e$ and $y^{e_1}$ associated if they are induced by the same $C$-ribbon,
$y$.
$$
\diagram
     & y\dlto\drto & \\
 y^e    &    & y^{e_1}
\enddiagram
$$


{\it Claim $2$.} $C_*^{e_1}$ contains all ribbons $y^{e_1}$ that are not associated to
                 $C_1^e$-ribbons.

By assumption, splicing $e$ produces the components $C_1^e,\ldots,C_k^e$. Let $V$
denote the set of external $C$-vertices that are, after slicing $e$, contained
in $C_2^e,\ldots,C_k^e$.
By Lemma~\ref{L:equivalence} components are characterized by $\sigma$-crossing,
whence slicing $e_1^e\in C_1^e$ does not affect any of the components
$C_2^e,\ldots,C_k^e$.
The key point here is that w.r.t.~$\sigma$-crossing, $e_1$ and $e_1^e$ have the same
effect since they differ only by $e_1$ being a $m$-ribbon and $e_1^e$ being a
$b$-ribbon, see Fig.~\ref{F:alter}.
Therefore slicing $e_1$ in $C$ does not affect the $\sigma$-crossing property of
$V$-vertices. By construction, $e$ connects all $V$-vertices which proves that
$C_*^{e_1}$ contains all ribbons $y^{e_1}$ that are not associated to
$C_1^e$-ribbons.

\begin{figure}[t]
\begin{center}
  \includegraphics[width=0.9\columnwidth]{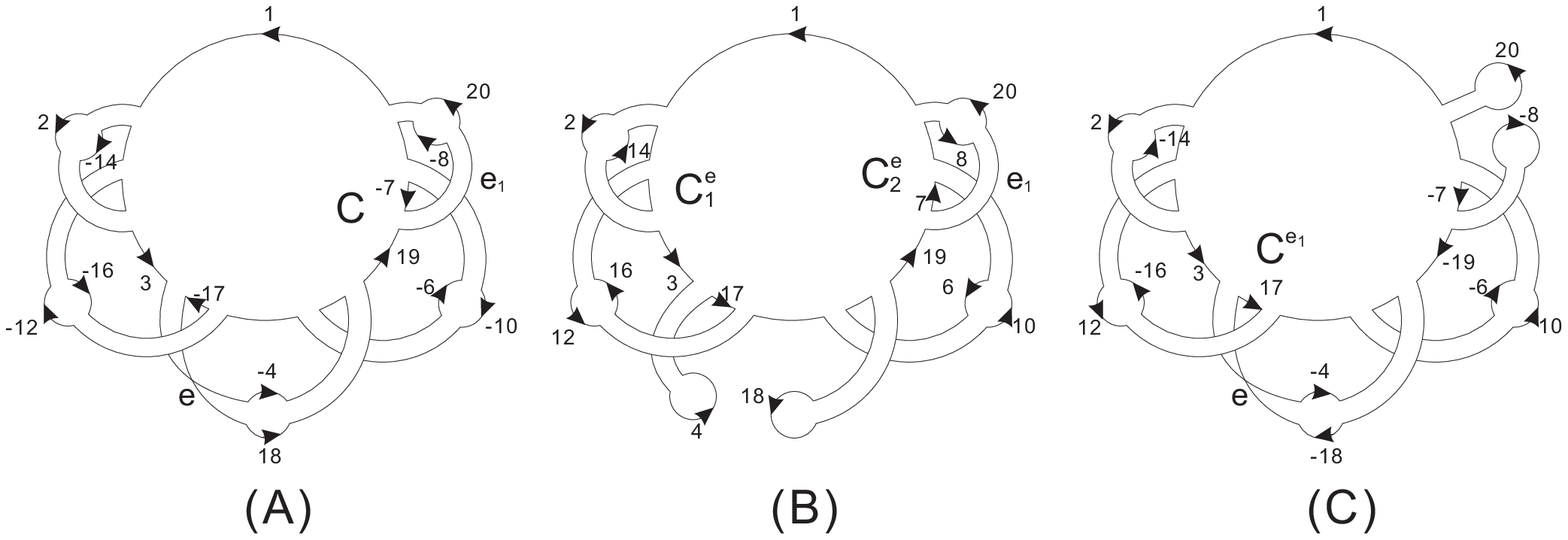}
\end{center}
\caption{\small (A) A non-orientable component, $C$, with $m$-ribbons
$e=((-3,17),(-4,18))$ and $e_1 = ((-7,19),(-8,20))$.
(B) Slicing $e$ in $C$ induces $C_1^{e}$ and $C_2^{e}$, both of which being orientable
components. The ribbon $e_1$ is contained in $C_2^{e}$.
(C) Slicing $e_1$ instead of $e$ in $C$ generates the component $C^{e_1}$. $C^{e_1}$
is non-orientable and contains all ribbons that are not associated to
$C_1^e$-ribbons.
}
\label{F:alter}
\end{figure}


Claim $1$ and Claim $2$ show that splicing $e_1$ instead of $e$ generates the
non-orientable component $C_*^{e_1}$, together with a set of ribbons, $z^{e_1}$,
that are associated to $C_1^e$-ribbons.
By assumption $C_1^e$ is an orientable component and since $e_1$ is spliced the
number of these $z^{e_1}$ is strictly smaller than the number of those contained
in $C_1^e$.
Accordingly, Claim $0$ follows by induction on the number of
ribbons contained in orientable components.

By Lemma~\ref{L:3} each such splicing reduces the genus by one eventually into a
$\pi$-map containing only external vertices of degree one. The lemma follows then from
Corollary~\ref{C:trivial}.
\end{proof}

As a result we now have

\begin{theorem}\label{T:lower}
Let $b_n$ be a signed permutation, then we have
$$
d(b_n)\ge g,
$$
i.e.~the topological genus is a sharp bound for the reversal distance.
\end{theorem}

\section{Orientable components}\label{S:orientable}

By Lemma~\ref{L:splicedown}, irreducible, non-orientable $\pi$-maps of genus $g$ have
reversal distance $g$. Thus it remains to analyze orientable components or equivalently,
orientable, irreducible $\pi$-maps.
In difference to non-orientable components, that could be treated individually,
orientable components acted upon by reversals have to be considered as an ensemble.
This is a result of Lemma~\ref{L:location}, i.e.~these components are either
concatenated or nested and the action of reversals affects entire chains of them.

We shall begin by showing that half-flipping transformes an orientable component into
an non-orientable component.

\begin{lemma} \label{L:twist_one}
Suppose $\mathbb{P}_{n}$ is a $\pi$-map and $C$ is a non-trivial, orientable
$\mathbb{P}_{n}$-component having genus $g$.
Let $e$ be a $b$-ribbon in $C$ having the two
even sectors, $x_{2j+2}, x_{2i}$ and let $\tilde{\mathbb{P}}_n$ be the $\pi$-map
obtained by half-flipping the vertex incident to $e$ w.r.t.~$x_{2i}$ and
$x_{2j+2}$.
Then $\tilde{C}$ is a non-orientable $\tilde{\mathbb{P}}_n$-component having genus
$g$.
\end{lemma}
\begin{proof}
Non-triviality of $C$ implies that $v$ has at least two incident ribbons and
as in Scenario $3$ we write $v$ as
\begin{eqnarray*}
v & = & (\underbrace{(\lambda_{x_{2j+2}},+)}_{x_{2j+2}}, h_1^{(1)}, \ldots, h_{m_1}^{(1)},
\underbrace{(\lambda_{x_{2i}},+)}_{x_{2i}}, h_1^{(2)}, \ldots, h_{m_2}^{(2)})\\
 & = & (\underbrace{(h_{m_2}^{(2)}, x_{2j+2})}_{e_1^{(1)}},
\underbrace{(x_{2j+2},  h_1^{(1)})}_{e_2^{(1)}}, e_3^{(1)}, \ldots, e_{m_1}^{(1)},
\underbrace{(h_{m_1}^{(1)}, x_{2i})}_{e_1^{(2)}},
\underbrace{(x_{2i},  h_1^{(2)})}_{e_2^{(2)}}, e_3^{(2)}, \ldots,
e_{m_2}^{(2)}).
\end{eqnarray*}
Since $C$ is orientable, all ribbons incident to $v$ are untwisted.
By half-flipping, $v$ becomes the vertex $\tilde{v}$,
$$
\tilde{v}=(\underbrace{(h_{m_2}^{(2)}, x_{2j+2})}_{e_1^{(1)}},
\underbrace{(x_{2i}, h_{m_1}^{(1)})}_{\overline{e}_1^{(2)}},
\overline{e}_{m_1}^{(1)}, \ldots, \overline{e}_3^{(1)},
\underbrace{(h_1^{(1)}, x_{2j+2})}_{\overline{e}_2^{(1)}},
\underbrace{(x_{2i}, h_1^{(2)})}_{e_2^{(2)}},
e_3^{(2)}, \ldots, e_{m_2}^{(2)}),
$$
where the ribbon $\overline{e}$ is obtained by twisting $e$.
By Lemma~\ref{L:equivalence}, a component is characterized by $(\bullet)$,
whence $\tilde{C}$ is a $\tilde{\mathbb{P}}_n$-component containing
the external vertex $\tilde{v}$. Furthermore $\tilde{v}$ is incident to both:
twisted and untwisted ribbons, respectively. By Lemma~\ref{L:orientable},
$\tilde{C}$ is non-orientable.
Euler's characteristic equation implies that $\tilde{C}$ has genus $g$.
\end{proof}

We now proceed by formalizing the partial order of orientable components
implied by Lemma~\ref{L:location}.
Let $\mathfrak{O}_{\mathbb{P}_n}$ denote the set of orientable components. By
Lemma~\ref{L:location} we have the partial order
$$
C \sqsubset C' \quad \Longleftrightarrow \quad  \text{\rm $C\neq C'$ and $C$ is
nested in $C'$.}
$$
We shall add to $(\mathfrak{O}_{\mathbb{P}_n},\sqsubset)$ the element $*$, which
contains any other orientable component. If we consider $(\mathfrak{O}_{\mathbb{P}_n}
\setminus \{C\},\sqsubset)$, we say $C$ is deleted from
$(\mathfrak{O}_{\mathbb{P}_n},\sqsubset)$. Let $C_1,C_2\in \mathfrak{O}_{\mathbb{P}_n}$ and
suppose $C_1,C_2\in \mathfrak{O}_{\mathbb{P}_n}$ are not nested.
Then we shall, w.l.o.g., assume
$$
\min_\sigma C_1 <_\sigma \max_\sigma C_1 <_\sigma \min_\sigma C_2.
$$
We set $C_{1,2}$ to be the smallest orientable component containing $C_1,C_2$
and
$$
[C_1,C_2]=\{ C_j \mid C_1\sqsubseteq  C_j \sqsubset C_2\}.
$$
Let $C\in \mathfrak{O}_{\mathbb{P}_n}$ such that $C_1 \sqsubset C$, $C_2 \sqsubset C$.
Then $C$ {\it separates} $C_1$ and $C_2$ if $C_1$ and $C_2$ are contained in distinct
$C$-gaps, see Fig.~\ref{F:contract} (A) and (B).
We write $C_1\lhd C_2$ if and only if
$$
{\max}_\sigma C_1 <_\sigma {\max}_\sigma C_2.
$$

\begin{figure}[t]
\begin{center}
  \includegraphics[width=0.7\columnwidth]{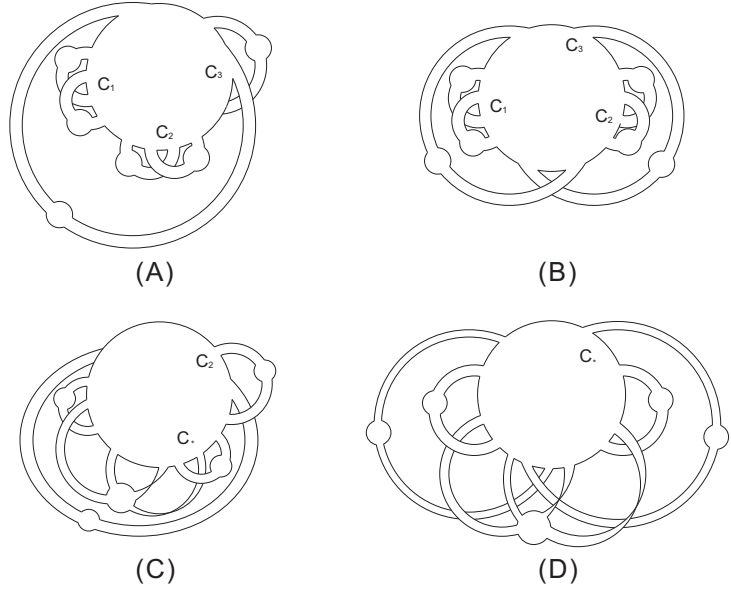}
\end{center}
\caption{\small Suppose $C_1 \lhd C_2 \lhd C_3$ and $C_1, C_2\sqsubset C_3=C_{1,2}$.
(A): $C_3$ does not separate $C_1$ and $C_2$, while in (B) $C_3$ does.
(C): gluing $C_1$ and $C_2$ in (A) does not merge $C_3$ into the
    non-orientable component $C_*$.
(D): gluing $C_1$ and $C_2$ in (B) merges $C_3$ into the non-orientable component
    $C_*$.
}
\label{F:contract}
\end{figure}


\begin{lemma}\label{L:contract} \cite{Pevzner:99}
Let $\mathbb{P}_n$ be a $\pi$-map with $C_1,C_2\in \mathfrak{O}_{\mathbb{P}_n}$ and
$x_{2j+2}$, $x_{2i}$ be two even sectors contained in $C_1$ and $C_2$, respectively.
Then gluing $x_{2j+2}$ and $x_{2i}$ generates a new, non-orientable component
$C_*$, obtained by merging the following set of orientable $\mathbb{P}_n$-components
\begin{equation}
M =
\begin{cases}
[C_1,C_{1,2}] \cup [C_2,C_{1,2}]  \cup \{C_{1,2}\}        & \text{\rm if $C_{1,2}$ separates
                                                      $C_1$ and $C_2$} \\
[C_1,C_{1,2}] \cup [C_2,C_{1,2}]  & \text{\rm otherwise. }
\end{cases}
\end{equation}
\end{lemma}
\begin{proof}
Let $C_1,C_2\in \mathfrak{O}_{\mathbb{P}_n}$ and let $v_1$ be the $C_1$-vertex containing
$x_{2j+2}$ and $v_2$ the $C_2$-vertex containing $x_{2i}$. By Lemma~\ref{L:model},
gluing merges $v_1$ and $v_2$ into $v_*$, without changing the $\sigma$-crossings of
any other vertices. Consequently, Lemma~\ref{L:model} and Lemma~\ref{L:equivalence}
imply that all $C_1$- and $C_2$-vertices merge in $\tilde{\mathbb{P}}_n$ into one
component, $C_*$.
By Lemma~\ref{L:model}, gluing produces a pair of sectors $\tilde{x}_{2j+2},
\tilde{x}_{2i}$, that are contained in a single, external vertex and that have different
orientations. Thus, by Lemma~\ref{L:orientable}, $C_*$, is a non-orientable component.

We consider $C_3\in [C_1,C_{1,2}]$ such that $C_1\sqsubset C_3$.
Since $C_1\sqsubset C_3$ there are two ribbons $(a,\sigma(a))$, $(b,\sigma(b))$ incident
to some $C_3$-vertex $v_3$ such that
\begin{equation}\label{E:cc1}
\sigma(a) <_\sigma d <_\sigma \sigma(d) <_\sigma b,
\end{equation}
where $d,\sigma(d)$ are the odd sectors of an arbitrary $C_1$-ribbon.
$C_2$ is by definition of $[C_1,C_{1,2}]$ not nested in $C_3$.
Thus $C_2$ lies, counterclockwise around $v^*$, to the right of $C_3$.
As a result we have
\begin{equation}\label{E:cc2}
\sigma(a) <_\sigma  \min_\sigma C_1 <_\sigma \max_\sigma C_1 <_\sigma b  <_\sigma
\min_\sigma C_2.
\end{equation}
Gluing $v_1,v_2$ w.r.t.~the sectors $x_{2j+2}$ and $x_{2i}$ generates the vertex $v_*$ and
a $\sigma$-crossing of $v_*$ and $v_3$. Therefore $C_3$ is merged into
$C_*$, see Fig.~\ref{F:contract} (D).

The case of $C_{1,2}$ separating $C_1$ and $C_2$ is analogous: then there exists a
$C_{1,2}$-vertex, $w$, together with two incident ribbons $(a,\sigma(a))$,
$(b,\sigma(b))$ such that eq.~(\ref{E:cc2}) holds. Accordingly, gluing $v_1,v_2$
w.r.t.~$x_{2j+2}$ and $x_{2i}$ generates a $\sigma$-crossing of $v_*$ and $w$,
see Fig.~\ref{F:contract} (C).
The case of $[C_2,C_{1,2}]$ is argued analogously, whence the lemma.
\end{proof}

Lemma~\ref{L:contract} suggests to glue two vertices contained in minimal
$\mathfrak{O}_{\mathbb{P}_n}$-elements, $C_1,C_2$. This collapses at least the
entire chains $[C_1,C_{1,2}]$ and $[C_2,C_{1,2}]$, respectively, as well as
possibly $C_{1,2}$ if $C_{1,2}$ separates $C_1,C_2$.

\section{The reversal distance}\label{S:hurdles}

This section is the reformulation of Hannenhalli and Pevzners treatment of hurdles
\cite{Pevzner:99} into the topological framework.
To relate our approach to breakpoint graphs, we note that an orientable component
in a $\pi$-map corresponds to a component without oriented cycles in the breakpoint
graph. Furthermore, gluing two hurdles in a $\pi$-map corresponds to the merging two
hurdles in the breakpoint graph \cite{Pevzner:99}.

\begin{definition}\label{D:hurdle}\cite{Pevzner:99}
A hurdle is either
\begin{itemize}
\item a minimal $\mathfrak{O}_{\mathbb{P}_n}$-element, i.e.~an interval
      $v^*$-$[i,k]_\sigma$ inducing an irreducible $\pi$-map, or
\item the maximum element in $(\mathfrak{O}_{\mathbb{P}_n},\sqsubset)$ which
      does not separate any pair of leaves.
\end{itemize}
A super-hurdle is a $(\mathfrak{O}_{\mathbb{P}_n},\sqsubset)$-hurdle, whose
deletion creates a $(\mathfrak{O}_{\mathbb{P}_n}\setminus\{C\},\sqsubset)$-hurdle.
\end{definition}

A reversal is called safe if it reduces $(g+h)$ by one, i.e., either $g$ decreases
by one and $h$ persists, or $g$ increases by one and $h$ decreases by two.

\begin{figure}[t]
\begin{center}
  \includegraphics[width=0.9\columnwidth]{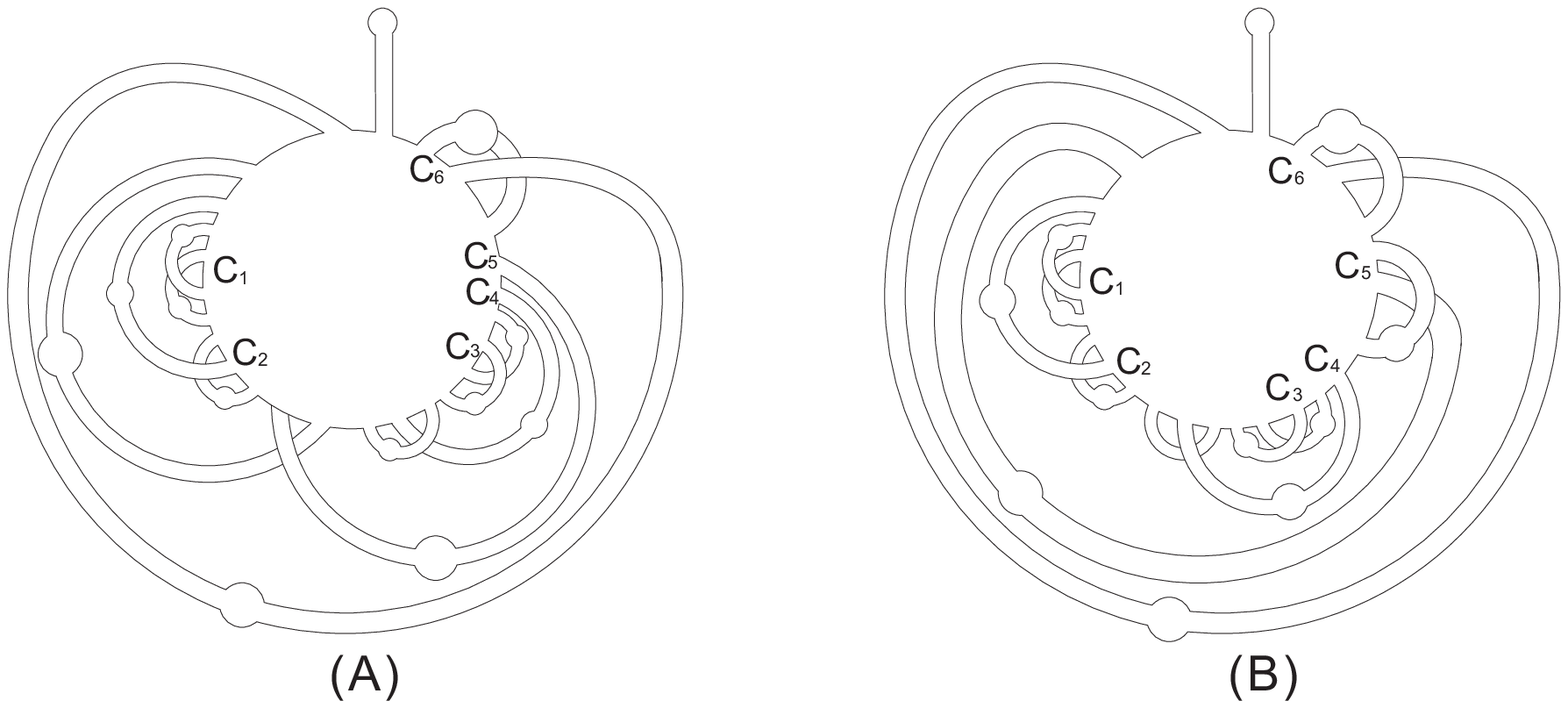}
\end{center}
\caption{\small Hurdles: (A) a $\pi$-map with $6$ orientable components.
$C_1$ and $C_3$ are minimal and $C_6$ is a maximal hurdle.
$C_1$ and $C_3$ are in addition super-hurdles, removing $C_1$ and $C_3$ renders
$C_2$ and $C_4$ as hurdle, respectively and neither, $C_2$ or $C_4$ were
originally a hurdle. $C_6$ is not a super-hurdle, since $C_5$ separates $C_1$
and $C_3$.
(B) $C_6$ is here a super-hurdle, since $C_5$ does not separate $C_1$ and $C_3$
    and becomes the maximum hurdle upon removal of $C_6$.
}
\label{F:hurdle}
\end{figure}


Then we have

\begin{lemma}\label{L:safe} \cite{Pevzner:99}
Let $\mathbb{P}_n$ be a $\pi$-map containing the two hurdles $H_1,H_2$ such that
$x_{2j+2}\in H_1$ and $x_{2i}\in H_2$ are two even sectors.
Then gluing w.r.t.~$x_{2j+2}$ and $x_{2i}$ is safe if there exist two hurdles
$U_1$ and $U_2$ such that
$$
H_1\lhd U_1\lhd H_2\lhd U_2 \quad \text{\rm or }\quad U_1\lhd H_1\lhd U_2\lhd H_2.
$$
\end{lemma}
\begin{proof}
Suppose that gluing the sectors $x_{2j+2}$ and $x_{2i}$ generates the $\pi$-map,
$\tilde{\mathbb{P}}_n$ and the $\mathfrak{O}_{\tilde{\mathbb{P}}_n}$-hurdle, $C$.
We shall distinguish the scenarios of (a) $C$ being minimal in
$\mathfrak{O}_{\tilde{\mathbb{P}}_n}$, or (b) $C$ being not minimal.

Ad (a): since $C$ is generated by the gluing of $H_1,H_2$, $C$ cannot have been
minimal in $\mathfrak{O}_{{\mathbb{P}}_n}$. Furthermore we have $H_1\sqsubset C$
and $H_2\sqsubset C$.
By Lemma~\ref{L:contract}, $[H_1,C]$ and $[H_2,C]$ collapse, merging into the
non-orientable component, $C_*$. In case of $H_1\lhd U_1\lhd H_2\lhd U_2$, we
have $U_1\sqsubset C$ and by construction $U_1$ does not merge into $C_*$.
This means that $C$ cannot be minimal in $\mathfrak{O}_{{\mathbb{P}}_n}$,
contradiction, see Fig.~\ref{F:super}, (a).
In case of $U_1\lhd H_1\lhd U_2\lhd H_2$ we argue analogously.

\begin{figure}[t]
\begin{center}
  \includegraphics[width=0.8\columnwidth]{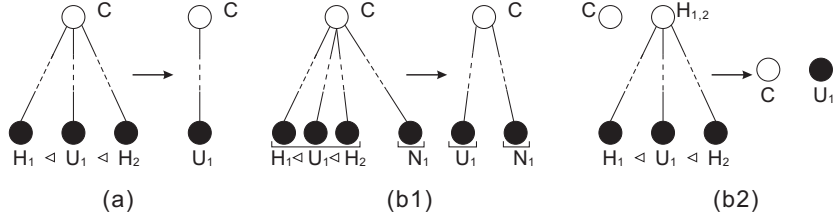}
\end{center}
\caption{\small Safe reversals when gluing $H_1$ and $H_2$:
an orientable component which is not a hurdle (white) and a hurdle (black).
(a): $C$ being minimal in $\mathfrak{O}_{\tilde{\mathbb{P}}_n}$, (b1): $C$ being maximal
and $H_1, H_2 \sqsubset C$ and (b2): $C$ being maximal and $H_1, H_2 \not\sqsubset C$.
}
\label{F:super}
\end{figure}

Ad (b): $C$ becomes the unique maximal element in $\mathfrak{O}_{\tilde{\mathbb{P}}_n}$.
We first observe that, by Lemma~\ref{L:contract}, in case of $H_1 \sqsubset C$ and
$H_2\not\sqsubset C$, $C$ necessarily merges into $C_*$. Accordingly, we have the
alternative:
$$
H_1,H_2\sqsubset C \quad \text{\rm  or}\quad  H_1,H_2\not\sqsubset C.
$$

In case of $H_1,H_2\sqsubset C$, $C$ does not separate $H_1,H_2$,
as it would vanish, otherwise.
In case of $H_1,H_2\not\sqsubset C$, $H_{1,2}$ necessarily
separates $H_1,H_2$, since otherwise, $H_{1,2}$ remains and
$H_{1,2}\not\sqsubset C$, whence $C$ is not the unique maximum.

(b1): in case of $H_1,H_2\sqsubset C$ and $C$ does not separate $H_1$ and $H_2$.
Since $C$ is not a hurdle in $\mathfrak{O}_{\mathbb{P}_n}$, $C$ necessarily
separates two orientable components.
As a result, there exists a hurdle $N\neq H_1,H_2$ such that $C$
separates $N$ and $H_1,H_2$.
In particular $C$ separates $N$ and $U_1$, where $H_1\lhd U_1\lhd H_2$.
The two hurdles $N$ and $U_1$ persist, respectively, when gluing $H_1,H_2$,
whence $C$ separates $N$ and $U_1$ in $\mathfrak{O}_{{\tilde{\mathbb{P}}}_n}$.
This implies that $C$ is not a hurdle in $\mathfrak{O}_{\tilde{{\mathbb{P}}}_n}$,
contradiction.

(b2): in case of $H_1,H_2\not\sqsubset C$ we can conclude that either $U_1$ or
$U_2$ are not nested in $C$. Thus $C$ cannot be the unique, maximal element
in $\mathfrak{O}_{\tilde{\mathbb{P}}_n}$, contraction.
\end{proof}

\begin{corollary}\label{C:safe_reversal}
Let $\mathbb{P}_{n,g}$ be a $\pi$-map and $h$ be the number of hurdles in
$\mathbb{P}_{n,g}$. Then in case of $h\neq 3$, there exists always a reversal
that acts safely on $\mathbb{P}_n$.
\end{corollary}
\begin{proof}
We label the hurdles such that $H_1\lhd H_2\lhd \cdots\lhd H_h$ holds.\\
For $h=1$, we half-flip, which preserves $g$ and reduces $h$ by one. Therefore,
any half-flip is in this scenario safe.\\
For $h=2$, we glue $H_1,H_2$. If $H_1,H_2$ are both minimal, then there exists no
unique maximal, non-separating hurdle. Thus $H_{1,2}$ is the unique maximal component
which separates $H_1,H_2$ and consequently vanishes by the gluing. Hence $g$ increases
by one and $h$ decreases by two.
If $H_2$ is the maximal hurdle, then we have a unique chain and gluing merges the
latter, whence $g$ increases by one and $h$ decreases by two. Thus, gluing $H_1$
and $H_2$ is a safe reversal.\\
For $h>3$, we glue $H_1$ and $H_{1+\lfloor h/2\rfloor}$, where
$$
H_1\lhd H_2\lhd H_{1+\lfloor h/2\rfloor} \lhd H_h.
$$
By Lemma~\ref{L:3} and Lemma~\ref{L:safe}, $g$ increases by
one and $h$ decreases by two, i.e.~it is a safe reversal.
By successively removing such pairs of hurdles, we can reduce the
situation to $h=1$ or $h=0$.
\end{proof}

\begin{lemma}\cite{Pevzner:99}
Suppose $\mathbb{P}_n$ is a $\pi$-map and $\tilde{\mathbb{P}}_n$ is obtained by gluing
the two hurdles $H_1$ and $H_2$. Then, any super-hurdle $U\neq H_1,H_2$ contained in
$(\mathfrak{O}_{\mathbb{P}_n},\sqsubset)$, is also a super-hurdle of
$(\mathfrak{O}_{\tilde{\mathbb{P}}_n},\sqsubset)$.
\end{lemma}
\begin{proof}
Since $U$ is a hurdle, by Lemma~\ref{L:contract}, gluing $H_1$ and $H_2$ does not merge
$U$ into $C_*$, whence $U\in (\mathfrak{O}_{\tilde{\mathbb{P}}_n},\sqsubset)$.
Since $U$ is a super-hurdle, deleting $U$ from $(\mathfrak{O}_{\mathbb{P}_n},\sqsubset)$
creates a new hurdle, $U'$.

The key point is to show that $U'\in (\mathfrak{O}_{\tilde{\mathbb{P}}_n},\sqsubset)$ and that
$U'$ is a hurdle in $\mathfrak{O}_{\tilde{\mathbb{P}}_n}$, i.e.~$U$ remains being a
super-hurdle in $(\mathfrak{O}_{\tilde{\mathbb{P}}_n},\sqsubset)$.
To establish $U'\in (\mathfrak{O}_{\tilde{\mathbb{P}}_n},\sqsubset)$ is a hurdle we consider
$U'$ in $(\mathfrak{O}_{\mathbb{P}_n}\setminus\{U\},\sqsubset)$ and distinguish two cases:

In case of $U'$ being minimal, we conclude that $U$ is minimal in $(\mathfrak{O}_{\mathbb{P}_n},
\sqsubset)$. Thus $U'$ is not contained in $[H_1,H_{1,2}]\cup [H_2,H_{1,2}]\cup \{H_{1,2}\}$,
i.e.~$U'$ remains to be a minimal hurdle in $(\mathfrak{O}_{\tilde{\mathbb{P}}_n},\sqsubset)$.

If $U'$ is the unique maximal hurdle in $(\mathfrak{O}_{\mathbb{P}_n}\setminus\{U\},\sqsubset)$,
it does not separate any pair of hurdles. We now inspect the effect of gluing $H_1$ and $H_2$.
Since $U'$ does not separate any pair of hurdles, $U'$ is not contained in $[H_1,H_{1,2}]\cup
[H_2,H_{1,2}]\cup \{H_{1,2}\}$ and accordingly not merged into $C_*$. Furthermore, $U'$ remains
to be the unique maximal hurdle in $(\mathfrak{O}_{\tilde{\mathbb{P}}_n},\sqsubset)$.

In both cases $U$ remains to be a super-hurdle in $(\mathfrak{O}_{\tilde{\mathbb{P}}_n},
\sqsubset)$ and the lemma follows.
\end{proof}

By Corollary~\ref{C:safe_reversal}, there exists always a safe reversal except in the
case $h=3$. If not all of them are super-hurdles, then half-flipping one non super-hurdle
reduces the situation to $h=2$. Thus it remains to consider the case of all three
hurdles being super-hurdles.

\begin{lemma}\label{L:3h} \cite{Pevzner:99}
Let $\mathbb{P}_n$ be a $\pi$-map with exactly three hurdles, all of which being
super-hurdles. Then there exists no safe reversal.
\end{lemma}
\begin{proof}
A safe reversal means to either (1) half-flip, which keeps $g$ and decreases $h$ by one, or
to (2) glue in which case $g$ increases by one and $h$ decreases by two.

In case of (1), since all three hurdles are super-hurdles, half-flipping is not safe
as it neither reduces the number hurdles nor the topological genus. Thus it suffices
to consider (2).\\

In case of (2), suppose we have the three super-hurdles denoted via
$H_1\lhd H_2\lhd H_3$ and that deleting $H_i$ creates the hurdle $U_i$,
$1\le i\le 3$. It follows from
$$
{\max}_\sigma H_1 <_\sigma {\max}_\sigma H_2<_\sigma {\max}_\sigma H_3,
$$
that $H_1$ and $H_2$ are necessarily minimal.

We distinguish two scenarios: \\

{\it Scenario 1.} suppose $H_3$ is minimal.\\
Then $\mathfrak{O}_{\mathbb{P}_n}$ has three chains $H_i, U_i, \ldots K_i$, $1\le i\le 3$, where
the $K_i$ are not necessarily all different. Furthermore, each chain has length at least two.
Assume that we glue the two minimum hurdles, say $H_1$ and $H_2$ (to glue $H_1$ and
$H_3$, $H_2$ and $H_3$ are argued similar). Then the two chains $H_1,U_1,\ldots,K_1$ and
$H_2,U_2,\ldots,K_2$ merge into $C_*$ either including $H_{1,2}$ or not. We distinguish
three cases: \\
{\rm (a)}  $H_{1,2}=K_1=K_2$ remains as a minimal element. Then $H_{1,2}$ is a hurdle and
           $g+h$ changes into $(g+1)+(h+1)-2=g+h$, \\
{\rm (b1)} $H_{1,2}$ is contracted and $K_3$ is not contracted. Then $K_3$ becomes the
           unique maximal hurdle (as it does not separate any two leaves) and $g+h$
           remains constant,\\
{\rm (b2)} $H_{1,2}=K_3$ is contracted. Then $K_1=K_2=K_3$ and $K_3$ separates
           $H_1$ and $H_2$. Since $H_3$ is a super-hurdle, deleting $H_3$ creates the
           hurdle, $U_3$, which implies $U_3\neq K_3$. Accordingly, gluing of $H_1,H_2$
           transforms the chain $H_3, U_3, \ldots, K_3$ into a chain having a maximal
           element not equal to $H_3$. The latter is a hurdle, whence $g+h$ remains constant.

{\it Scenario 2.} suppose $H_3$ is maximal. \\
Then there are the two chains
$$
H_1, U_1, \ldots, U_3, H_3, \quad \quad H_2, U_2, \ldots, U_3, H_3,
$$
where $U_1, U_2\neq U_3$ and $H_3$, as a super-hurdle, does not separate $H_1$ and $H_2$.
Then gluing $H_1$ and $H_2$ is not safe as it preserves $H_3$ and produces an unique
chain, creating a new, minimal hurdle.
Using Lemma~\ref{L:equivalence}, we observe that gluing $H_1$ and $H_3$ produces an
unique chain together with a new, maximal hurdle. Therefore, the number of hurdles
decreases only by one, whence it is not safe.

Accordingly, we have shown that there exists no safe reversal in a $\pi$-map which contains
only three hurdles, all of which are super-hurdles.
\end{proof}

\begin{theorem}
Let $b_n$ be a signed permutation with associated class of $\pi$-maps $[\mathbb{P}_{n,g}]$,
having genus $g$ and $h$ hurdles. Then we have
$$
d(b_n) =
\begin{cases}
g+h+1 & \text{if $h\neq 1$ is odd and all $h$ hurdles are super-hurdles,} \\
g+h  & \text{otherwise}  \\
\end{cases}
$$
\end{theorem}
\begin{proof}
If $h$ is even, we label the hurdles $H_1\lhd H_2\lhd  \cdots \lhd H_{2k}$.
In case of $k=1$, we have $h=2$ and gluing $H_1$ and $H_2$ is safe.
In case of $k>1$, we glue $H_1$ and $H_{k+1}$, which is safe by Lemma~\ref{L:safe}.
Iterating this $h/2$ times we obtain a $\pi$-map with genus
$g'=g+h/2$ without any hurdles, i.e., a $\pi$-map in which each component is non-orientable.
All these non-orientable components can be reduced to trivial ones via $g'$ slicings.
Accordingly, the total number of reversals, i.e.~$d(b_n)$, is $h/2+g'=h/2 + g+h/2 = g+h$.

If $h$ is odd and $\mathbb{P}_{n,g}$ contains at least one hurdle, which is not a
super-hurdle, then we apply a half-flip, deriving a $\pi$-map of genus $g$ having $(h-1)$
hurdles. This reduces this case to the case of $h$ being even and the total number of
reversals is $1+ (g+h-1) = g+h$.

Finally, suppose $h$ is odd and all hurdles are super-hurdles,
$H_1\lhd H_2\lhd  \cdots \lhd H_{2k+1}$.
By Lemma~\ref{L:safe} we have safe reversals and can reduce the situation to a scenario
of exactly three super-hurdles. By Lemma~\ref{L:3h} there exists no safe reversal, then.
Gluing any pair of these three creates a new hurdle and reduces the scenario to that of
$h=2$. The number of reversal in this case is
$\frac{h-1}{2} + (g+\frac{h-1}{2}) + 1 + 1 = g+h+1$.
\end{proof}

\section{Acknowledgments.}
We acknowledge the financial support of the Future and Emerging
Technologies (FET) programme within the Seventh Framework Programme (FP7) for
Research of the European Commission, under the FET-Proactive grant agreement
TOPDRIM, number FP7-ICT-318121.

\bibliographystyle{elsarticle-num}
\bibliography{signed}

\begin{thebibliography}{10}
\expandafter\ifx\csname url\endcsname\relax
  \def\url#1{\texttt{#1}}\fi
\expandafter\ifx\csname urlprefix\endcsname\relax\def\urlprefix{URL }\fi
\expandafter\ifx\csname href\endcsname\relax
  \def\href#1#2{#2} \def\path#1{#1}\fi

\bibitem{Pevzner:99}
S.~Hannenhalli, P.~A. Pevzner, Transforming cabbage into turnip: polynomial
  algorithm for sorting signed permutations by reversals, J. ACM 46(1) (1999)
  1--27.

\bibitem{Watterson:1982}
G.~A. Watterson, W.~J. Ewens, T.~E. Hall, A.~Morgan, The chromosome inversion
  problem, J. Theoret. Biol. 99 (1982) 1--7.

\bibitem{Nadeau:1984}
J.~H. Nadeau, B.~A. Taylor, Lengths of chromosomal segments conserved since
  divergence of man and mouse, Proc. Natl. Acad. Sci. USA 81 (1984) 814--818.

\bibitem{Bafna:1996}
V.~Bafna, P.~A. Pevzner, Genome rearrangements and sorting by reversals, SIAM
  J. Comput. 25~(2) (1996) 272--289.

\bibitem{Kaplan:97}
H.~Kaplan, R.~Shamir, R.~E. Tarjan, Faster and simpler algorithm for sorting
  signed permutations by reversals, in: SIAM Journal of Computing, 1997, pp.
  344--351.

\bibitem{Bader:2001}
D.~A. Bader, B.~M.~E. Moret, M.~Yan, A linear-time algorithm for computing
  inversion distance between signed permutations with an experimental study, J.
  Comp. Biol. 8(5) (2001) 483--491.

\bibitem{Caprara:1997}
A.~Caprara, Sorting by reversals is difficult, in: Proceedings of the First
  Annual International Conference on Computational Molecular Biology, 1997, pp.
  75--83.

\bibitem{Christie:1998}
D.~A. Christie, A 3/2-approximation algorithm for sorting by reversals, in:
  Proceedings of the Ninth Annual ACM-SIAM Symposium on Discrete Algorithms,
  1998, pp. 244--252.

\bibitem{Kececioglu:95}
J.~Kececioglu, D.~Sankoff, Exact and approximation algorithms for sorting by
  reversals, with application to genome rearrangement, Algorithmica 13~(1-2)
  (1995) 180--210.

\bibitem{Penner:10}
R.~C. Penner, M.~Knudsen, C.~Wiuf, J.~E. Andersen, Fatgraph models of proteins,
  Comm.\ Pure Appl.\ Math. 63 (2010) 1249--1297.

\bibitem{Massey:69}
W.~S. Massey, Algebraic Topology: An Introduction, Springer-Veriag, New York,
  1967.

\end{thebibliography}

\end{document}